\documentclass{amsart}

\usepackage{enumerate,xypic,mathdots,mathrsfs,verbatim}

\usepackage{fullpage}
\usepackage{amssymb}

\DeclareFontFamily{U}{mathx}{\hyphenchar\font45}
\DeclareFontShape{U}{mathx}{m}{n}{
      <5> <6> <7> <8> <9> <10>
      <10.95> <12> <14.4> <17.28> <20.74> <24.88>
      mathx10
      }{}
\DeclareSymbolFont{mathx}{U}{mathx}{m}{n}
\DeclareFontSubstitution{U}{mathx}{m}{n}
\DeclareMathAccent{\widecheck}{0}{mathx}{"71}
\DeclareMathAccent{\wideparen}{0}{mathx}{"75}

\numberwithin{equation}{section}

\newtheorem{prop}{Proposition}[section]
\newtheorem{theorem}[prop]{Theorem}
\newtheorem{cor}[prop]{Corollary}
\newtheorem{lemma}[prop]{Lemma}

\theoremstyle{definition}
\newtheorem{defn}[prop]{Definition}
\newtheorem{example}[prop]{Example}

\theoremstyle{remark}
\newtheorem{rem}[prop]{Remark}

\newcommand{\A}{\mathscr{A}}

\newcommand{\C}{\mathbb{C}}
\newcommand{\D}{\mathcal{D}}
\newcommand{\F}{\mathcal{F}}
\newcommand{\G}{\mathcal{G}}
\renewcommand{\H}{\mathcal{H}}

\newcommand{\N}{\mathbb{N}}

\newcommand{\Q}{\mathbb{Q}}
\newcommand{\R}{\mathbb{R}}

\newcommand{\T}{\mathbb{T}}
\newcommand{\Z}{\mathbb{Z}}

\newcommand{\Norm}[1]{\left\Vert #1 \right\Vert}

\renewcommand{\subset}{\subseteq}

\DeclareMathOperator{\spn}{span}

\DeclareMathOperator{\supp}{supp}

\begin{document}

\title{Subspaces of $L^2(G)$ invariant under translation by an abelian subgroup}
\author{Joseph W.\ Iverson}
\address{Department of Mathematics, University of Oregon, Eugene, OR 97403--1222, USA}
\email{iverson@uoregon.edu}
\date{\today}


\maketitle

\begin{abstract}
For a second countable locally compact group $G$ and a closed abelian subgroup $H$, we give a range function classification of closed subspaces in $L^2(G)$ invariant under left translation by $H$. For a family $\A \subset L^2(G)$, this classification ties with a set of conditions under which the translations of $\A$ by $H$ form a continuous frame or a Riesz sequence. When $G$ is abelian, our work relies on a fiberization map; for the more general case, we introduce an analogue of the Zak transform. Both transformations intertwine translation with modulation, and both rely on a new group-theoretic tool: for a closed subgroup $\Gamma \subset G$, we produce a measure on the space $\Gamma \backslash G$ of right cosets that gives a measure space isomorphism $G \cong \Gamma \times \Gamma \backslash G$. Outside of the group setting, we consider a more general problem: for a measure space $X$ and a Hilbert space $\H$, we investigate conditions under which a family of functions in $L^2(X;\H)$ multiplies with a basis-like system in $L^2(X)$ to produce a continuous frame or a Riesz sequence in $L^2(X;\H)$. Finally, we explore connections with dual integrable representations of LCA groups, as introduced by Hern{\'a}ndez et al.\ in \cite{HSWW}.
\end{abstract}

\medskip

%INTRODUCTION===========================================================
\section{Introduction} \label{sec:intro}

The goal of this paper is to give a comprehensive analysis of translation invariance from the new vantage point of the Zak transform. Let $G$ be a second countable locally compact group, and let $H\subset G$ be a closed abelian subgroup. Given a function $f\colon G \to \C$ and $y \in G$, we will write $L_y f\colon G \to \C$ for the \emph{left translation} given by
\[ (L_y f)(x) = f(y^{-1}x). \]
A closed subspace $M\subset L^2(G)$ is called \emph{$H$-translation-invariant}, or $H$-TI, if $L_\xi f \in M$ whenever $f\in M$ and $\xi \in H$. Our main result classifies $H$-TI spaces in terms of range functions and an abstract version of the Zak transform. This is followed by a set of conditions under which the left $H$-translates of a family $\A \subset L^2(G)$ form a continuous frame or a Riesz sequence. For the subgroup $\Z^n \subset \R^n$, similar results using the so-called fiberization map were given by de Boor, Devore, and Ron \cite{BDR1,BDR2} and Bownik \cite{B}, with wide applications to the theories of approximation, spline systems, wavelets, and Gabor systems. This work was generalized to the LCA setting by Kamyabi Gol and Raisi Tousi \cite{KR}; Cabrelli and Paternostro \cite{CP}; and Bownik and Ross \cite{BR}. Attacking from another direction, the ``extra'' invariance of an invariant subspace, the authors of \cite{ACHKM,ACP,ACP2,SW} describe $H$-TI spaces under the assumption that $H$ contains a nice subgoup $K$. All of these papers require that $G$ be abelian and that $G/H$ be compact. By replacing fiberization with a kind of Zak transform, we can go far beyond these assumptions. Indeed, we will allow $G$ to be any (possibly nonabelian) second countable locally compact group, and we will ask only that $H$ be closed and abelian. This grants us access to a number of basic examples that were previously inaccessible, like $\Z^m \subset \R^n$, corresponding to shifts in $L^2(\R^n)$ by a lattice of less-than-full rank, and $\R^m \subset \R^n$. It also opens the door to a world of more sophisticated examples, such as the non-normal copy of $\R$ in the $ax+b$ group. 

All previous research on $H$-TI spaces has required the larger group $G$ to be abelian, and analysis has relied on the fiberization map. In effect, fiberization applies the Fourier transform of the larger group $G$ and then splits off the annihilator subgroup $H^*\subset \hat{G}$. The main innovation of this paper is to replace fiberization with an abstract version of the Zak transform, which reverses this procedure, first splitting off the subgroup $H$, and then applying \emph{its} Fourier transform. By shifting the Fourier transform from the big group to the little group, we dispense with the need for commuting assumptions on $G$. Moreover, we can eliminate any remaining topological assumptions on $H$ by defining our Zak transform in terms of a new group-theoretic tool, which is of interest in its own right. This accounts for the greater generality of our results.

While the Zak transform has been widely used in the abelian setting for analysis of shift-modulation invariant subspaces, its use in the classification of $H$-TI spaces is appearing here for the first time.\footnote{While putting the finishing touches on this paper, the author learned that Barbieri, Hern{\'a}ndez, and Paternostro \cite{BHP2} simultaneously and independently developed a classification of $H$-TI spaces using an identical notion of Zak transform, with many of the same results as appear in our Sections \ref{sec:Zak} and \ref{sec:HTIStr}. The hypotheses in \cite{BHP2} differ slightly from ours: their encompassing group $G$ is assumed to be abelian throughout, and their version of our Theorem \ref{thm:HTIFrame} assumes the generating family $\A$ is countable. While our papers have nontrivial intersection, each contains a considerable amount of material not present in the other.} Indeed, our Zak transform analysis gives a new classification of shift-invariant subspaces in $L^2(\R)$. Moreover, we give new conditions for the integer shifts of a family $\A \subset L^2(\R)$ to form a frame or a Riesz sequence. These results have immediate applications for wavelets and multi-resolution analysis.

In addition to our work with the Zak transform, we are able to finish the story on fiberization for a second countable LCA group $G$ with a closed subgroup $H$. The previous development of fiberization required $G/H$ to be compact. By using our new group-theoretic tool, we can remove this assumption. The result is a fiberization analysis of $H$-TI spaces that works in a wide variety of previously inaccessible situations, including the subgroups $\Z^m \subset \R^n$ and $\R^m \subset \R^n$ mentioned above.

The paper is organized as follows. In Section \ref{sec:frames}, we investigate a common source of frames and Riesz bases in a measure-theoretic setting. Let $X$ be a measure space, and let $\H$ be a separable Hilbert space. Any two functions $f\colon X \to \C$ and $\varphi \colon X \to \H$ can be multiplied pointwise, with product $(f \varphi)(x) = f(x)\varphi(x)$. Given a basis-like set $\D \subset L^2(X)$ and a family $\A \subset L^2(X;\H)$, we give conditions under which $\{f \varphi : f \in \D, \varphi \in \A\}$ forms a continuous frame or a Riesz sequence in $L^2(X;\H)$. The main results here are Theorems \ref{thm:Riesz} and \ref{thm:frame}, relating frame and Riesz sequence conditions in $L^2(X;\H)$ to the corresponding pointwise conditions in $\H$. This section is meant as a companion to the recent work by Bownik and Ross \cite[\S 2]{BR} on range functions and multiplicative invariance in the measure-theoretic setting.

In Section \ref{sec:meas}, we develop a version of Weil's formula for right cosets. Given a closed subgroup $\Gamma \subset G$, we produce a measure on the space $\Gamma \backslash G$ that allows for a measure space isomorphism $G \cong \Gamma \times \Gamma \backslash G$. This isomorphism is the key to our development of the abstract Zak transform and fiberization map in Section \ref{sec:Zak}. There we give a wide variety of examples, and describe connections between the Zak transform and the fiberization map in the abelian setting. 

Both fiberization and the Zak transform are put to use in Section \ref{sec:HTIStr}, where we give our main results classifying $H$-TI spaces and describing conditions under which the left $H$-translates of a family of functions in $L^2(G)$ form a continuous frame or a Riesz sequence. At the end of this section, we analyze translation/modulation-invariant spaces under critical sampling in the abelian setting.

In Section \ref{sec:DualInt}, we consider the related problem of invariant subspaces for dual integrable representations of locally compact abelian (LCA) groups, introduced by Hern{\'a}ndez, \v{S}iki{\'c}, Weiss, and Wilson in \cite{HSWW}. We show that dual integrable representations are precisely those gotten from the translation action of an abelian subgroup, as in Section \ref{sec:HTIStr}. We then give a range function classification of invariant subspaces. Moreover, we explain when the orbit of a family of vectors produces a continuous frame. If the group is discrete, we do the same for Riesz sequences. Our results generalize those in \cite{HSWW}, which treated discrete LCA groups and cyclic subspaces.

\medskip

%ABSTRACT FRAMES=======================================================
\section{Frames and Riesz bases in $L^2(X;\H)$}\label{sec:frames}

Let $\A$ be a countable family of functions in $L^2(\R^n)$, and let
\[ E(\A) = \{T_k f : k \in \Z^n, f \in \A\},\]
where $T_k f(y) = f(y-k)$. In \cite{B}, Bownik studied $E(\A)$ through the fiberization map $\mathcal{T} \colon L^2(\R^n) \to L^2([0,1)^n; l^2(\Z^n))$ given by
\[ \mathcal{T}f(x) = ( \hat{f}(x + k))_{k\in \Z^n}. \]
The utility of $\mathcal{T}$ comes from the intertwining relation
\[ \mathcal{T} (T_k f)(x) = e^{2\pi i k \cdot x} \mathcal{T} f(x) \quad \text{for all }k\in \Z^n, \]
so that integer shifts in $L^2(\R^n)$ become modulations by an orthonormal basis of $L^2([0,1)^n)$ in $L^2([0,1)^n; l^2(\Z^n))$. Taking advantage of this correspondence, Bownik gave sufficient and necessary conditions for $E(\A)$ to form a frame or a Riesz basis for its closed linear span. Later, Cabrelli and Paternostro \cite{CP} and Kamyabi Gol and Raisi Tousi \cite{KR} generalized this method to the setting of a second countable LCA group $G$ with a closed discrete subgroup $H$ such that $G/H$ is compact. Bownik and Ross \cite{BR} have gone even further by removing the hypothesis that $H$ be discrete, replacing frames with so-called continuous frames. Each of these papers achieves its goal by transforming $L^2(G)$ into a space $L^2(X;\H)$, with $X$ a measure space and $\H$ a Hilbert space, in such a way that translations by the subgroup $H\subset G$ become modulations by a nice family of functions in $L^\infty(X)$. In this section, we consider the latter situation more generally. Namely, for a countable family $\A \subset L^2(X;\H)$ and a basis-like family $\D$ of functions on $X$, we investigate conditions under which the family of functions
\[ (g \varphi)(x) = g(x)\cdot \varphi(x), \qquad g \in \D,\, \varphi \in \A\]
form a continuous frame or a Riesz basis for their closed linear span. This work is complementary to a recent publication by Bownik and Ross \cite[\S 2]{BR} extending Helson's theory of multiplicative invariance \cite{H2}. We now describe their main results. 

\begin{defn}\label{def:DetRan}
Let $(X,\mu)$ be a measure space. A \emph{determining set} for $L^1(X)$ is a family of functions $\D \subset L^\infty(X)$ with the property that, for all $f\in L^1(X)$,
\[ \int_X f(x) g(x)\, d\mu(x) = 0 \quad \text{for all }g\in \D \implies f = 0. \]
Given a separable Hilbert space $\H$, a closed subspace $M \subset L^2(X;\H)$ is said to be $\D$\emph{-multiplication invariant}, or $\D$-MI, if for every $g \in \D$ and every $\varphi \in M$, the function $(g\varphi)(x) = g(x) \varphi(x)$ also belongs to $M$. Given a family $\A \subset L^2(X;\H)$, we denote
\[ S_\D(\A) = \overline{\spn}\{ g \varphi \colon g \in \D, \varphi \in \A\} \]
for the $\D$-MI space it generates, and
\[ E_\D(\A) = \{ g \varphi \colon g \in \D, \varphi \in \A\}. \]
We will consider $E_\D(\A)$ as a set with multiplicities.

A \emph{range function} is a mapping
\[ J \colon X \to \{ \text{closed subspaces of }\H\}. \]
For a range function $J$ and $x \in X$, we denote $P_J(x) \colon \H \to \H$ for the orthogonal projection onto $J(x)$. We say that $J$ is a \emph{measurable} range function if, for each $(u,v) \in \H \times \H$, the function $x \mapsto \langle P_J(x) u, v \rangle$ is measurable on $X$.
\end{defn}

Range functions have a long history in the classification of invariant subspaces, dating at least as far back as Helson \cite{H} and Srinivasan \cite{S} in 1964. More recently, Bownik \cite{B} used range functions to classify shift invariant subspaces of $L^2(\R^n)$. This program was continued in increasing generality by Cabrelli and Paternostro \cite{CP}, Kamyabi Gol and Raisi Tousi \cite{KR}, Currey, Mayeli, and Oussa \cite{CMO}, and Bownik and Ross \cite{BR}. Our results in Sections \ref{sec:HTIStr} and \ref{sec:DualInt} continue this line of research. 

The proposition below is a slight modification of Theorem 2.4 in \cite{BR}. See also Srinivasan \cite[Theorem 3]{S} and Helson \cite[Lecture VI, Theorem 8]{H} and \cite[Ch.\ 1, \S 3, Theorem 1]{H2}.

\begin{prop}\label{prop:AbstRan}
Let $(X,\mu)$ be a $\sigma$-finite measure space, let $\D$ be a determining set for $L^1(X)$, and let $\H$ be a separable Hilbert space.
\begin{enumerate}[(i)]
\item If $J\colon X \to \{\text{closed subspaces of }\H\}$ is a range function, then
\[ M_J = \{ \varphi \in L^2(X; \H) : \varphi(x) \in J(x) \text{ for a.e. }x \in X\} \]
is a closed $\D$-MI subspace of $L^2(X; \H)$.
\item The correspondence $J \mapsto M_J$ is a bijection between measurable range functions and closed $\D$-MI subspaces of $L^2(X;\H)$, provided we identify range functions that agree a.e.\ on $X$. 
\item Let $\A$ be a family of functions in $L^2(X;\H)$, let $\A_0 \subset \A$ be a countable dense subset, and let $J$ be the range function defined almost everywhere by
\begin{equation}\label{eq:AbstRan1}
J(x) = \overline{\spn}\{ \varphi(x) : \varphi \in \A_0\}.
\end{equation}
Then
\begin{equation}\label{eq:AbstRan2}
M_J = S_\D(\A_0) = S_\D(\A).
\end{equation}
\end{enumerate}
\end{prop}

\begin{proof}
In \cite{BR}, Proposition 2.1 and Theorem 2.4 prove everything except the fact that $S_\D(\A_0) = S_\D(\A)$. One inclusion in this equality is obvious. For the other, let $\varphi \in \A$ be arbitrary, and let $\{\varphi_k\}_{k=1}^\infty$ be a sequence in $\A_0$ with $\varphi_k \to \varphi$. By passing to a subsequence if necessary, we may assume that $\varphi_k(x) \to \varphi(x)$ a.e. Since $J$ maps $X$ into the set of closed subspaces, $\varphi(x) \in J(x)$ for a.e.\ $x\in X$. In other words, $\varphi \in M_J = S_\D(\A_0)$. Since this holds for every $\varphi \in \A$, $S_\D(\A) \subset S_\D(\A_0)$.
\end{proof}

\smallskip

%Riesz sequences
\subsection{Riesz sequences}

We remind the reader that a countable family $(u_i)_{i \in I}$ of vectors in a Hilbert space $\H$ is called a \emph{Riesz sequence} if there are constants $0< A \leq B < \infty$ such that for all $(c_i)_{i\in I} \in l^2(I)$ with only finitely many $c_i \neq 0$,
\begin{equation}\label{eq:Riesz}
A \sum_{i \in I} |c_i|^2 \leq \Norm{ \sum_{i\in I} c_i u_i }^2 \leq B \sum_{i \in I} |c_i|^2.
\end{equation}
The constants $A$ and $B$ are called \emph{bounds}. When a Riesz sequence spans a dense subspace of $\H$, it is called a \emph{Riesz basis}.

The following theorem is an abstract version of \cite[Theorem 2.3(ii)]{B}. Our proof is a modification of the argument given there. See also \cite[Theorem 4.3]{CP} and \cite[Theorem 5.1]{BR}.

\begin{theorem}\label{thm:Riesz}
Let $(X,\mu)$ be a measure space with $\mu(X) < \infty$, and let $\H$ be a Hilbert space. For a countable family $\A \subset L^2(X;\H)$ and constants $0 < A \leq B < \infty$, the following are equivalent:
\begin{enumerate}[(i)]

\item For some orthonormal basis $\D$ of $L^2(X)$, $\{g \varphi : g \in \D, \varphi \in \A\}$ is a Riesz sequence in $L^2(X;\H)$ with bounds $A,B$.

\item For any Riesz sequence $(g_i)_{i\in I}$ in $L^2(X)$ with bounds $a,b$, $\{ g_i \varphi : i \in I, \varphi \in \A\}$ is a Riesz sequence in $L^2(X;\H)$ with bounds $aA$, $bB$.

% For any orthonormal basis $\D$ of $L^2(X)$, $\{g \varphi : g \in \D, \varphi \in \A\}$ is a Riesz sequence in $L^2(X;\H)$ with bounds $A,B$.

\item For any family $(f_\varphi)_{\varphi \in \A} \subset L^2(X)$ having only finitely many $f_\varphi \neq 0$,
\begin{equation}\label{eq:Riesz1}
A \sum_{\varphi \in \A} \int_X | f_\varphi(x) |^2\, d\mu(x) \leq \int_X \Norm{ \sum_{\varphi \in \A} f_\varphi(x) \varphi(x) }^2 d\mu(x) \leq B \sum_{\varphi \in \A} \int_X | f_\varphi(x) |^2\, d\mu(x).
\end{equation}

\item For any family $(f_\varphi)_{\varphi \in \A} \subset L^2(X)$ with $\sum_{\varphi \in \A} \Norm{ f_\varphi }_{L^2(X)}^2 < \infty$, \eqref{eq:Riesz1} holds.

\item For almost every $x\in X$, $\{\varphi(x) : \varphi \in \A \}$ is a Riesz sequence in $\H$ with bounds $A,B$.

\end{enumerate}
\end{theorem}

Condition (iii) can be read as a strictly stronger version of the usual definition of Riesz sequence, where the coefficient sequence $(c_\varphi)_{\varphi \in \A} \in \l^2(\A)$ has been replaced with a function sequence $(f_\varphi)_{\varphi \in \A} \in l^2(\A;L^2(X))$. As a consequence of Proposition \ref{prop:AbstRan}, the theorem remains valid if we replace every instance of ``Riesz sequence'' with ``Riesz basis''.

\smallskip

\noindent \emph{Proof of Theorem \ref{thm:Riesz}}.
(iii) $\implies$ (v).
Suppose there is a set $Y \subset X$ of positive measure such that, for each $x \in Y$, $\{ \varphi(x) : \varphi \in \A\}$ is \emph{not} a Riesz sequence in $\H$ because it fails the upper bound of \eqref{eq:Riesz}. We'll show that \eqref{eq:Riesz1} fails in the upper bound. Let $\{d_m\}_{m=1}^\infty$ be a dense subset of $l^2(\A)$ such that each $d_m = (d_{m,\varphi})_{\varphi \in \A}$ has only finitely many nonzero entries. For each $m,n \in \N$, put
\[ E_{m,n} = \{ x \in X : \Norm{ \sum_{\varphi \in \A} d_{m,\varphi} \varphi(x) }^2 > (B + \frac{1}{n}) \sum_{\varphi \in \A} |d_{m,\varphi}|^2\}, \]
which is well-defined up to a set of measure zero. If $x\notin \bigcup_{m,n=1}^\infty E_{m,n}$, then
\[ \Norm{ \sum_{\varphi \in \A} d_{m,\varphi} \varphi(x) }^2 \leq B \sum_{\varphi \in \A} |d_{m,\varphi}|^2 \quad \text{for all }m, \]
so $\{ \varphi(x): \varphi \in \A\}$ satisfies the upper bound of the Riesz condition with bound B. Consequently, one of the sets $E_{m,n}$ has $\mu(E_{m,n}) > 0$, and for this $m$ and $n$ we define a family of functions $(f_\varphi)_{\varphi \in \A}\subset L^2(X)$ by the formula $f_\varphi(x) = d_{m,\varphi} \mathbf{1}_{E_{m,n}}(x)$. Only finitely many of these functions are nonzero, yet
\[ \int_X \Norm{ \sum_{\varphi \in \A} f_\varphi(x) \varphi(x) }^2 d\mu(X) = \int_X \mathbf{1}_{E_{m,n}}(x) \Norm{ \sum_{\varphi \in \A} d_{m,\varphi} \varphi(x) }^2 d\mu(X) \]
\[  \geq \int_X \mathbf{1}_{E_{m,n}}(x)\cdot (B + \frac{1}{n}) \sum_{\varphi \in \A} | d_{m,\varphi} |^2\, d\mu(x) = (B + \frac{1}{n}) \int_X \sum_{\varphi \in \A} | d_{m,\varphi} \mathbf{1}_{E_{m,n}}(x)|^2\, d\mu(x) \]
\[ = (B + \frac{1}{n}) \sum_{\varphi \in \A} \int_X |f_\varphi(x)|^2\, d\mu(x). \]
Thus (iii) fails in the upper bound.  In other words, the upper bound in (iii) implies the upper bound in (v). A similar argument applies for the lower bounds.

(v) $\implies$ (ii).
Let $(g_i)_{i\in I}$ be a Riesz sequence in $L^2(X)$ with bounds $a,b$, and let $(c_{i,\varphi})_{i\in I, \varphi \in \A} \in \l^2(I \times \A)$ be a sequence having only finitely many nonzero terms. If (v) holds, then for a.e.\ $x \in X$ we apply the Riesz condition with the sequence $( \sum_{i \in I} c_{i,\varphi} g_i(x) )_{\varphi \in \A}$ to deduce
\[ A \sum_{\varphi \in \A} \left| \sum_{i \in I} c_{i,\varphi} g_i(x) \right|^2 \leq \Norm{ \sum_{\varphi \in \A} \left( \sum_{i \in I} c_{i,\varphi} g_i(x) \right) \varphi(x) }_{\H}^2 \leq B \sum_{\varphi \in \A} \left| \sum_{i \in I} c_{i,\varphi} g_i(x) \right|^2. \]
Integrating this inequality over $X$ produces
\begin{equation} \label{eq:Riesz3}
A \sum_{\varphi \in \A} \Norm{ \sum_{i\in I} c_{i,\varphi} g_i }_{L^2(X)}^2 \leq \Norm{ \sum_{\varphi \in \A} \sum_{i \in I} c_{i,\varphi} g_i \varphi }_{L^2(X;\H)}^2 \leq B \sum_{\varphi \in \A} \Norm{ \sum_{i\in I} c_{i,\varphi} g_i }_{L^2(X)}^2.
\end{equation}
Meanwhile, for any $\varphi \in \A$ we apply the Riesz condition in $L^2(X)$ to deduce
\[ a \sum_{i\in I} |c_{i,\varphi}|^2 \leq \Norm{ \sum_{i\in I} c_{i,\varphi} g_i }_{L^2(X)}^2 \leq b \sum_{i\in I} |c_{i,\varphi}|^2. \]
Adding over all $\varphi \in \A$ and combining with \eqref{eq:Riesz3} gives
\[ aA \sum_{\varphi \in \A} \sum_{i\in I} |c_{i,\varphi}|^2 \leq \Norm{ \sum_{\varphi \in \A} \sum_{i \in I} c_{i,\varphi} g_i \varphi }_{L^2(X;\H)}^2 \leq bB \sum_{\varphi \in \A} \sum_{i\in I} |c_{i,\varphi}|^2. \]
In other words, $\{ g_i \varphi : i \in I, \varphi \in \A\}$ is a Riesz sequence with bounds $aA,bB$.

(ii) $\implies$ (i). This is immediate.

(i) $\implies$ (iii).
Suppose $L^2(X)$ has an orthonormal basis $\D$ for which $\{g \varphi : g \in \D, \varphi \in \A\}$ is a Riesz sequence with bounds $A,B$. First, let $(p_\varphi)_{\varphi \in \A} \subset L^2(X)$ be a family in the finite linear span of $\D$, with only finitely many $p_\varphi \neq 0$. Writing
\[ p_\varphi(x) = \sum_{g\in \D} c_{g,\varphi} g(x) \]
for appropriate constants $c_{g,\varphi}$, we have (by definition)
\[ \int_X \Norm{ \sum_{\varphi \in \A} p_{\varphi}(x) \varphi(x) }^2 d\mu(x) = \Norm{ \sum_{\varphi \in \A} \sum_{g\in \D} c_{g,\varphi} g \varphi }^2 \]
and (by Parseval's identity)
\[ \sum_{\varphi \in \A} \int_X | p_{\varphi}(x) |^2\, d\mu(x) = \sum_{\varphi \in \A} \sum_{g \in \D} |c_{g,\varphi} |^2. \]
Since $\{g \varphi : g \in \D, \varphi \in \A\}$ is a Riesz sequence, and since only finitely many $c_{g,\varphi} \neq 0$,
\begin{equation}\label{eq:Riesz5}
A \sum_{\varphi \in \A} \int_X | p_\varphi(x) |^2\, d\mu(x) \leq \int_X \Norm{ \sum_{\varphi \in \A} p_\varphi(x) \varphi(x) }^2 d\mu(x) \leq B \sum_{\varphi \in \A} \int_X | p_\varphi(x) |^2\, d\mu(x),
\end{equation}

Now let $(f_\varphi)_{\varphi \in \A}\subset L^2(X)$ be a family of functions as in (iii). For each $\varphi \in \A$, there is a sequence $\{p_{\varphi,k}\}_{k=1}^\infty$ of functions in the finite linear span of $\D$ such that $p_{\varphi,k} \to f_\varphi$ in $L^2(X)$. By passing to a subsequence if necessary, we may assume that $p_{\varphi,k}(x) \to f_{\varphi}(x)$ almost everywhere on $X$. Moreover, we can assume that $p_{\varphi,k}=0$ when $f_\varphi = 0$, so that for each $k$, only finitely many $p_{\varphi,k}\neq 0$. By Fatou's Lemma and \eqref{eq:Riesz5},
\[ \int_X \Norm{ \sum_{\varphi \in \A} f_\varphi(x) \varphi(x) }^2 d\mu(x) \leq \liminf_{k\to \infty} \int_X \Norm{ \sum_{\varphi\in \A} p_{\varphi,k}(x) \varphi(x) }^2 d\mu(x) \leq \liminf_{k\to \infty}\, B \sum_{\varphi \in \A} \int_X | p_{\varphi,k}(x) |^2\, d\mu(x) \]
\[ = B \sum_{\varphi \in \A} \int_X |f_\varphi(x)|^2\, d\mu(x). \]
In other words, the upper bound holds in \eqref{eq:Riesz1}. 

It remains to prove the lower bound. To do this, we will upgrade the first inequality above to an equality. Previously, we showed that the upper bound in (iii) implies the upper bound in (v). Thus, for any sequence $(c_\varphi)_{\varphi \in \A} \in l^2(\A)$ having only finitely many nonzero entries,
\[ \Norm{ \sum_{\varphi \in \A} c_\varphi \varphi(x) }^2 \leq B \sum_{\varphi \in \A} |c_\varphi|^2 \quad \text{for a.e. }x\in X. \]
In particular, $\Norm{\varphi(x)}^2 \leq B$ for all $\varphi \in \A$ and a.e.\ $x\in X$. Therefore,
\[ \int_X \Norm{f_\varphi(x) \varphi(x) - p_{\varphi,k}(x) \varphi(x)}^2 d\mu(x) = \int_X | f_\varphi(x) - p_{\varphi,k}(x)|^2 \Norm{\varphi(x)}^2 d\mu(x) \]
\[ \leq B \int_X | f_\varphi(x) - p_{\varphi,k}(x) |^2\, d\mu(x). \]
Since $p_{\varphi,k} \to f_\varphi$ in $L^2(X)$, $p_{\varphi,k} \varphi \to f_\varphi \varphi$ in $L^2(X;\H)$. In particular,
\[ \int_X \Norm{ p_{\varphi,k}(x) \varphi(x) }^2 d\mu(x) \to \int_X \Norm{ f_\varphi(x) \varphi(x) }^2 d\mu(x). \]
Now \eqref{eq:Riesz1} follows from \eqref{eq:Riesz5}.

(iii) $\iff$ (iv).
Obviously (iv) implies (iii). Suppose conversely that (iii) holds. Without loss of generality, we may assume that $\A$ is infinite, and then we can enumerate $\A = \{\varphi_k\}_{k=1}^\infty$.  Let $\{f_k\}_{k=1}^\infty \subset L^2(X)$ be a sequence of functions such that $\sum_{k=1}^\infty \Norm{f_k}_{L^2(X)}^2 < \infty$. By Tonelli's Theorem,
\[ \int_X \sum_{k=1}^\infty | f_k(x) |^2\, d\mu(x) = \sum_{k=1}^\infty \int_X | f_k(x) |^2\, d\mu(x) < \infty, \]
so $(f_k(x))_{k=1}^\infty \in l^2(\N)$ for a.e.\ $x\in X$. We have shown that (iii) implies (v). Thus $(\varphi_k(x))_{k=1}^\infty$ is a Riesz sequence for a.e.\ $x\in X$. Applying the synthesis operator, we find the sum $\sum_{k=1}^\infty f_k(x) \varphi_k(x)$ converges unconditionally for a.e.\ $x\in X$.

For each $n\in \N$, (iii) gives
\begin{equation}\label{eq:Riesz2}
A \sum_{k=1}^n \int_X |f_k(x)|^2\, d\mu(x) \leq \int_X \Norm{ \sum_{k=1}^n f_k(x) \varphi_k(x) }^2 d\mu(x) \leq B \sum_{k=1}^n \int_X |f_k(x)|^2.
\end{equation}
Moreover, Fatou's Lemma and another application of (iii) show that
\[ \int_X \Norm{ \sum_{k=1}^\infty f_k(x) \varphi_k(x) - \sum_{k=1}^n f_k(x) \varphi_k(x) }^2 d\mu(x) = \int_X \lim_{N\to \infty} \Norm{ \sum_{k=n+1}^N f_k(x) \varphi_k(x) }^2 d\mu(x) \]
\[ \leq \liminf_{N\to \infty} \int_X \Norm{ \sum_{k=n+1}^N f_k(x) \varphi_k(x) }^2 d\mu(x) \leq \liminf_{N\to \infty}\, B \sum_{k=n+1}^N \int_X |f_k(x)|^2\, d\mu(x) = B \sum_{k=n+1}^\infty \Norm{f_k}_{L^2(X)}^2. \]
Thus $\sum_{k=1}^n f_k \varphi_k \to \sum_{k=1}^\infty f_k \varphi_k$ in $L^2(X;\H)$-norm. Completeness shows that $\sum_{k=1}^\infty f_k \varphi_k \in L^2(X;\H)$, and continuity of the norm gives
\[ \lim_{n\to \infty} \int_X \Norm{ \sum_{k=1}^n f_k(x) \varphi_k(x) }^2 d\mu(x) = \int_X \Norm{ \sum_{k=1}^\infty f_k(x) \varphi_k(x) }^2 d\mu(x). \]
Sending $n\to \infty$ in \eqref{eq:Riesz2} establishes (iv). \hfill \qed

\smallskip

\begin{rem} \label{rem:Riesz}
The theorem above holds when $\mu(X) = \infty$, but only vacuously. Indeed, the hypothesis that $\mu(X) < \infty$ was never used. However, if $\A \subset L^2(X;\H)$ is any countable family satisfying (v), then $\Norm{\varphi(x)}^2 \geq A$ for every $\varphi \in \A$ and a.e.\ $x\in X$. Hence,
\[ \int_X \Norm{\varphi(x)}^2 d\mu(x) \geq A\, \mu(X) = \infty \]
for each $\varphi \in \A$, so that $\A = \emptyset$.
\end{rem}

\smallskip

%frames
\subsection{Continuous frames}

\begin{defn}
Let $\H$ be a Hilbert space, and let $(\mathcal{M},\mu_\mathcal{M})$ be a measure space. A family of vectors $(u_t)_{t\in \mathcal{M}} \subset \H$ is called a \emph{continuous frame} over $\mathcal{M}$  for $\H$ if both of the following hold:
\begin{enumerate}[(i)]
\item For each $v\in \H$, the function $t \mapsto \langle v, u_t \rangle$ is measurable $\mathcal{M} \to \C$.
\item There are constants $0 < A \leq B < \infty$, called \emph{frame bounds}, such that for each $v \in \H$,
\[ A \Norm{v}^2 \leq \int_\mathcal{M} \left| \langle v, u_t \rangle \right|^2\, d\mu_\mathcal{M}(t) \leq B \Norm{v}^2. \]
\end{enumerate}
When $A=B$, the frame is called \emph{tight}, and when $A=B=1$, it is a \emph{continuous Parseval frame}.
\end{defn}

In practice, it is enough to check condition (ii) for $v$ in a dense subset of $\H$. See \cite[Proposition 2.5]{RND}. Continuous frames were introduced independently by Kaiser \cite{Ka} and Ali, Antoine, and Gazou \cite{AAG}. When $\mathcal{M}$ is a countable set and $\mu_\mathcal{M}$ is counting measure, continuous frames reduce to the usual discrete version.

\begin{defn}
Let $(X,\mu_X)$ be a measure space. A \emph{Parseval determining set} for $L^1(X)$ consists of another measure space $(\mathcal{M},\mu_\mathcal{M})$ and a family of functions $(g_t)_{t\in \mathcal{M}} \subset L^\infty(X)$ such that for each $f\in L^1(X)$, the mapping
\[ t \mapsto \int_X f(x) \overline{g_t(x)}\, d\mu_X(x) \]
is measurable on $\mathcal{M}$, and
\begin{equation}\label{eq:ParDet}
\int_\mathcal{M} \left| \int_X f(x) \overline{g_t(x)}\, d\mu_X(x) \right|^2 d\mu_\mathcal{M}(t) = \int_X | f(x) |^2\, d\mu_X(x).
\end{equation}
We allow that both sides may be infinite.
\end{defn}

This definition axiomatizes Lemma 3.5 of \cite{BR}. If $(g_t)_{t\in \mathcal{M}}$ is a Parseval determining set for $L^1(X)$, then so is $( \overline{g_t})_{t\in \mathcal{M}}$, by taking complex conjugates in the integrands above. It follows easily that a Parseval determining set  is a determining set in the sense of Definition \ref{def:DetRan}.

In the sections that follow, our primary example of a Parseval determining set will be the characters of an LCA group; see Lemma \ref{lem:CharParDet} infra. For another example, suppose that $X$ is equipped with counting measure. If $\mathcal{M}$ is a countable set with counting measure, then a family $(g_t)_{t\in \mathcal{M}} \subset l^2(X) \subset l^\infty(X)$ is a Parseval determining set for $L^1(X)$ if and only if
\[ \sum_{t\in \mathcal{M}} \left| \sum_{x\in X} f(x) \overline{g_t(x)} \right|^2 = \sum_{x\in X} |f(x)|^2 \quad \text{for all }f\in l^1(X) \subset l^2(X), \]
if and only if $(g_t)_{t\in \mathcal{M}}$ is a discrete Parseval frame for $l^2(X)$.

If we relax our conditions and allow $(X,\mu)$ to be an arbitrary $\sigma$-finite measure space, then any Parseval determining set $(g_t)_{t\in \mathcal{M}} \subset L^\infty(X) \cap L^2(X)$ satisfies the Parseval condition on the dense subspace $L^1(X) \cap L^2(X) \subset L^2(X)$, so it is a Parseval frame for $L^2(X)$. However, not every Parseval frame for $L^2(X)$ consisting of functions in $L^\infty(X)$ is a Parseval determining set for $L^1(X)$.\footnote{The author thanks Prof.\ Alexander Olevskii for his help answering this question.} Indeed, for any $f \in L^1([0,1]) \setminus L^2([0,1])$ there is an orthonormal basis $(g_n)_{n=1}^\infty \subset C([0,1])$ for $L^2([0,1])$ such that $\int f \overline{g_n} = 0$ for every $n$; see \cite[Satz 613]{KS}. For these functions, the left hand side of \eqref{eq:ParDet} is infinite, but the right hand side is zero.

\begin{defn}
Let $(X,\mu_X)$ and $(\mathcal{M},\mu_\mathcal{M})$ be measure spaces, and let $\H$ be a separable Hilbert space. We say a family $(\varphi_t)_{t\in \mathcal{M}} \subset L^2(X;\H)$ is \emph{jointly measurable} if there is a function $\Phi \colon \mathcal{M} \times X \to \H$ satisfying the conditions:
\begin{enumerate}[(i)]
\item For a.e. $t\in \mathcal{M}$, $\Phi(t,\cdot) = \varphi_t$ a.e.\ on $X$.
\item For any $u\in \H$, the function $(t,x) \mapsto \langle \Phi(t,x), u \rangle$ is measurable on $\mathcal{M} \times X$.
\end{enumerate}
\end{defn}

By Pettis's Measurability Theorem \cite[Theorem 1.1]{P}, condition (ii) says precisely that $\Phi \colon \mathcal{M} \times X \to \H$ is measurable with respect to the Borel $\sigma$-algebra on $\H$. In the case where $\H = L^2(Y)$, this is equivalent to another kind of pointwise measurability property, which we describe in Corollary \ref{cor:JointMeas} below.

Intuitively, joint measurability means that the function $(t,x) \mapsto \varphi_t(x)$ is measurable on $\mathcal{M} \times X$. However, this notion may depend on the choice of representative functions $\varphi_t \colon X \to \H$. In the sequel, we will often ignore this subtlety and integrate expressions involving $\varphi_t(x)$ over $\mathcal{M} \times X$. When this happens, it is to be assumed that we have fixed a measurable function $\Phi \colon \mathcal{M} \times X \to \H$ as above.

We expect that the next proposition is already known. However, we have not been able to locate a reference. Therefore, we supply a proof.

\begin{prop} \label{prop:JointMeas}
Let $(\mathcal{M},\mu_{\mathcal{M}})$ and $(Y,\mu_Y)$ be $\sigma$-finite measure spaces, with $(\mathcal{M},\mu_{\mathcal{M}})$ complete. For a family $(f_t)_{t\in \mathcal{M}} \subset L^2(Y)$, the following are equivalent:
\begin{enumerate}[(i)]
\item There is a measurable function $F \colon \mathcal{M} \times Y \to \C$ such that, for a.e.\ $t\in \mathcal{M}$, $F(t,\cdot) = f_t$ a.e.\ on $Y$.
\item For each $g\in L^2(Y)$, the function $t \mapsto \langle f_t, g \rangle$ is measurable on $\mathcal{M}$.
\end{enumerate}
\end{prop}

\begin{proof}
First assume that (i) holds. Find a sequence of simple measurable functions $S_n \colon \mathcal{M} \times Y \to \C$ such that $S_n(t,y) \to F(t,y)$ for all $(t,y) \in \mathcal{M}$, with $|S_n(t,y)| \leq |F(t,y)|$. Using the $\sigma$-finite conditions, we may assume that each $S_n$ has support contained in a measurable rectangle with finite measure. For every $g\in L^2(Y)$ and every $n$, H\"older's Inequality yields
\[ \int_\mathcal{M} \int_Y | S_n(t,y) \overline{g(y)}|\, d\mu_Y(y)\, d\mu_\mathcal{M}(t) \leq \int_\mathcal{M} \left( \int_Y | S_n(t,y) |^2\ d\mu_Y(y) \right)^{1/2} \left( \int_Y |g(y)|^2\ d\mu_Y(y) \right)^{1/2}\, d\mu_\mathcal{M}(t) < \infty. \]
Therefore Fubini's Theorem applies to the function $(t, y) \mapsto S_n(t,y) \overline{g(y)}$, and in particular the function
\[ t \mapsto \int_Y S_n(t,y) \overline{g(y)}\, d\mu_Y(y) \]
is well defined a.e.\ and measurable on $\mathcal{M}$. Now the Lebesgue Dominated Convergence Theorem shows that
\[ \langle f_t, g \rangle = \int_Y F(t,y) \overline{g(y)}\, d\mu_Y(y) = \lim_{n\to \infty} \int_Y S_n(t,y) \overline{g(y)}\, d\mu_Y(y) \]
for a.e.\ $t\in \mathcal{M}$. Hence the function $t \mapsto \langle f_t, g \rangle$ is the a.e.\ pointwise limit of measurable functions, and is itself measurable.

Suppose conversely that (ii) holds. By Pettis's Measurability Theorem, the function $t \mapsto f_t$ is measurable $\mathcal{M} \to L^2(Y)$; hence $t \mapsto \Norm{f_t}$ is measurable on $\mathcal{M}$. An easy exercise now shows that the measurable space $\mathcal{M}$ admits another measure $\tilde{\mu}_\mathcal{M}$ for which
\[ \int_\mathcal{M} \Norm{f_t}^2\, d\tilde{\mu}_\mathcal{M}(t) < \infty. \]
Since we are concerned only with measurability, we may replace $\mu_\mathcal{M}$ with $\tilde{\mu}_\mathcal{M}$ and assume that the function $t \mapsto f_t$ belongs to $L^2(\mathcal{M};L^2(Y))$. The usual identification of $L^2(\mathcal{M} \times Y)$ with $L^2(\mathcal{M}; L^2(Y))$ now proves (i). (See for instance \cite[Theorem II.10(c)]{RS2}.)
\end{proof}

\begin{cor} \label{cor:JointMeas}
Let $(X,\mu_X)$, $(Y,\mu_Y)$, and $(\mathcal{M},\mu_\mathcal{M})$ be complete, $\sigma$-finite measure spaces. A family $(\varphi_t)_{t \in \mathcal{M}} \subset L^2(X;L^2(Y))$ is jointly measurable if and only if there is a measurable function $\Phi\colon \mathcal{M} \times X \times Y \to \C$ such that for a.e.\ $t\in \mathcal{M}$, for a.e.\ $x\in X$, $\Phi(t,x,\cdot) = \varphi_t(x)$ a.e.\  on $Y$. Consequently, the notion of ``joint measurability'' remains unchanged when we identify $L^2(X;L^2(Y))$ with $L^2(X\times Y) = L^2(X\times Y; \C)$, or with $L^2(Y;L^2(X))$.
\end{cor}

\begin{proof}
Apply Proposition \ref{prop:JointMeas} to the family $(\varphi_t(x))_{(t,x) \in \mathcal{M} \times X} \subset L^2(Y)$. We leave it to the reader to check the details surrounding sets of measure zero.
\end{proof}

The next theorem is an abstract version of \cite[Theorem 2.3(i)]{B}, whose argument we follow. See also \cite[Theorem 4.2]{CP} and \cite[Theorem 5.1]{BR}.

\begin{theorem}\label{thm:frame}
Let $(X,\mu_X)$ and $(\mathcal{M},\mu_\mathcal{M})$ be $\sigma$-finite measure spaces, and let $\D = (g_s)_{s\in \mathcal{M}}$ be a Parseval determining set for $L^1(X)$. Fix a separable Hilbert space $\H$, another $\sigma$-finite measure space $(\mathcal{N},\mu_\mathcal{N})$, and a jointly measurable familiy $\A = (\varphi_t)_{t\in \mathcal{N}} \subset L^2(X;\H)$. Let $\A_0 \subset \A$ be a countable dense subset, and define $J$ as in \eqref{eq:AbstRan1}. For constants $0< A \leq B < \infty$, the following are equivalent:

\begin{enumerate}[(i)]
\item $E_\D(\A)$ forms a continuous frame for $S_\D(\A)$ over $\mathcal{M} \times \mathcal{N}$, with bounds $A,B$. That is,
\[ A \int_X \Norm{\psi(x)}^2 d\mu_X(x) \leq \int_\mathcal{N} \int_\mathcal{M} \left| \int_X \langle \psi(x), g_s(x) \varphi_t(x) \rangle\, d\mu_X(x) \right|^2 d\mu_\mathcal{M}(s)\, d\mu_\mathcal{N}(t) \leq B \int_X \Norm{\psi(x)}^2 d\mu_X(x) \]
for all $\psi \in S_\D(\A)$.

\item For a.e.\ $x\in X$ and every $u \in J(x)$,
\[ A \Norm{u}^2 \leq \int_\mathcal{N} | \langle u, \varphi_t(x) \rangle |^2 d\mu_\mathcal{N}(t) \leq B \Norm{u}^2. \]
\end{enumerate}
\end{theorem}

We are tempted to interpret condition (ii) to mean that the family $\{ (\varphi_t)(x) : t\in \mathcal{N}\}$ forms a continuous frame for $J(x)$ for a.e.\ $x \in X$. However, when $\A$ is uncountable, the vectors $\varphi_t(x)$ need not reside in $J(x)$. A more precise interpretation says that $\{ P_J(x)[ \varphi_t(x)] : t \in \mathcal{N}\}$ forms a continuous frame for $J(x)$ for a.e.\ $x \in X$.

The theorem significantly reduces the problem of determining when $E_\D(\A)$ forms a continuous frame. For instance, when $\A \subset L^2(X;\H)$ is a countable family equipped with counting measure, condition (ii) says that for a.e.\ $x\in X$, $\{ \varphi(x) : \varphi \in \A\}$ forms a discrete frame for $J(x)$. Thus, a continuous problem in $L^2(X; \H)$ reduces to a discrete problem in $\H$. The reduction is even more pronounced when $\A$ consists of a single function $\varphi\in L^2(X;\H)$. In that case, (ii) reduces to
\begin{enumerate}[(i')]
\setcounter{enumi}{1}
\item For a.e.\ $x \in X$, either $\varphi(x) = 0$ or $A \leq \Norm{\varphi(x)}^2 \leq B$.
\end{enumerate}

\smallskip

\noindent \emph{Proof of Theorem \ref{thm:frame}.} 
Joint measurability of $\A$ ensures that the integrals above are well defined; use Tonelli's Theorem for the integral in condition (ii). For each $\psi \in S_\D(\A)$, we compute
\begin{equation}\label{eq:frame}
\int_\mathcal{N} \int_\mathcal{M} \left| \int_X \langle \psi(x), g_s(x) \varphi_t(x) \rangle\, d\mu_X(x) \right|^2 d\mu_\mathcal{M}(s)\, d\mu_\mathcal{N}(t)
\end{equation}
\[ = \int_\mathcal{N} \int_\mathcal{M} \left| \int_X \langle \psi(x), \varphi_t(x) \rangle \overline{g_s(x)} \, d\mu_X(x) \right|^2 d\mu_\mathcal{M}(s) d\mu_\mathcal{N}(t) \]
\[ = \int_\mathcal{N} \int_X | \langle \psi(x), \varphi_t(x) \rangle |^2\, d\mu_X(x)\, d\mu_\mathcal{N}(t) = \int_X \int_\mathcal{N} | \langle \psi(x), \varphi_t(x) \rangle |^2\, d\mu_\mathcal{N}(t)\, d\mu_X(x), \]
since $\D$ is a Parseval determining set for $L^1(X)$.

If (ii) holds, then \eqref{eq:AbstRan2} shows that
\[ A \Norm{ \psi(x) }^2 \leq \int_\mathcal{N} | \langle \psi(x), \varphi_t(x) \rangle |^2\, d\mu_\mathcal{N}(t) \leq B \Norm{\psi(x) }^2 \]
for all $\psi \in S_\D(\A)$ and a.e.\ $x\in X$. Integrating over $X$ and applying \eqref{eq:frame} proves (i).

Suppose conversely that (ii) fails. Fix a countable dense subset $\{u_m\}_{m=1}^\infty \subset \H$. For a.e.\ $x\in X$, it follows that $\{P_J(x) u_m\}_{m=1}^\infty$ is a dense subset of $J(x)$. Given $m,n \in \N$, define
\begin{align*}
E_{m,n} &= \{ x \in X : \int_\mathcal{N} | \langle P_J(x) u_m, \varphi_t(x) \rangle |^2\, d\mu_\mathcal{N}(t) > (B+\frac{1}{n}) \Norm{P_J(x) u_m}^2 \} \\
F_{m,n} &= \{ x \in X : \int_\mathcal{N} | \langle P_J(x) u_m, \varphi_t(x) \rangle |^2\, d\mu_\mathcal{N}(t) < (A-\frac{1}{n}) \Norm{P_J(x) u_m}^2 \},
\end{align*}
each of which is well-defined up to a set of measure zero. For a.e.\ $x \notin \bigcup_{m,n=1}^\infty (E_{m,n} \cup F_{m,n})$, 
\[ A \Norm{P_J(x)u_m}^2 \leq \int_\mathcal{N} | \langle P_J(x) u_m, \varphi_t(x) \rangle |^2\, d\mu_\mathcal{N}(t) \leq B \Norm{P_J(x)u_m}^2 \quad \text{for all }m \in \N, \]
so that $\{ P_J(x) [\varphi_t(x)]\}_{t \in \mathcal{N}}$ forms a frame for $J(x)$ with bounds $A,B$. Therefore at least one set $E_{m,n}$ or $F_{m,n}$ has positive measure. In the first case, fix a Borel set $E\subset E_{m,n}$ with $0 < \mu_X(E) < \infty$, and define $\theta \in L^2(X; \H)$ by
\[ \theta(x) = \mathbf{1}_E(x)\cdot  P_J(x) u_m. \]
Since we used strict inequality in the definition of $E_{m,n}$, $\Norm{\theta(x)} > 0$ on $E$. Moreover, $\theta \in S_\D(\A)$ by \eqref{eq:AbstRan2}, and \eqref{eq:frame} shows that
\[ \int_\mathcal{N} \int_\mathcal{M} \left| \int_X \langle \theta(x), g_s(x) \varphi_t(x) \rangle\, d\mu_X(x) \right|^2 d\mu_\mathcal{M}(s)\, d\mu_\mathcal{N}(t) = \int_X \int_\mathcal{N} | \langle \theta(x), \varphi_t(x) \rangle |^2\, d\mu_\mathcal{N}(t)\, d\mu_X(x) \]
\[ = \int_X \mathbf{1}_E(x)\cdot \int_\mathcal{N} | \langle P_J(x) u_m, \varphi_t(x) \rangle |^2\, d\mu_\mathcal{N}(t)\, d\mu_X(x) \geq \int_X \mathbf{1}_E(x) \cdot (B + \frac{1}{n}) \Norm{ P_J(x) u_m }^2 d\mu_X(x) \]
\[ = (B + \frac{1}{n}) \int_X \Norm{ \theta(x) }^2 d\mu_X(x). \]
Thus (i) fails. A similar argument shows that (i) fails when $\mu_X(F_{m,n}) > 0$. This completes the proof. \hfill \qed

\medskip

%MEASURES==============================================================
\section{A Weil formula for right cosets} \label{sec:meas}

Let $G$ be a second countable locally compact group, and let $\Gamma\subset G$ be a closed subgroup. We emphasize that these groups need not be abelian. Our purpose is to examine the measure-theoretic interplay between $G$, $\Gamma$, and the topological quotients $G/\Gamma$ and $\Gamma\backslash G$; the latter is the space of \emph{right} cosets of $\Gamma$ in $G$. Our main result is the existence of a measure on $\Gamma\backslash G$ for which $G \cong \Gamma \times \Gamma\backslash G$ as measure spaces, and for which the resulting unitary $U \colon L^2(G) \to L^2(\Gamma\times \Gamma\backslash G)$ is well behaved under left translation by $\Gamma$.

There is a positive regular Borel measure $\mu_G$ on $G$, called \emph{(left) Haar measure}, such that
\begin{equation}\label{eq:Haar}
\int_G f(yx)\, d\mu_G(x) = \int_G f(x)\, d\mu_G(x)
\end{equation}
for all $f \in L^1(G,\mu_G)$ and all $y \in G$. This measure is unique up to multiplication by a scalar $c > 0$. Fix a scale once and for all. Equation \eqref{eq:Haar} generally fails if we replace $yx$ with $xy$. However, there is a continuous function $\Delta_G \colon G \to (0,\infty)$, called the \emph{modular function}, such that
\begin{equation}\label{eq:ModularR}
\int_G f(xy)\, d\mu_G(x) = \Delta_G(y^{-1}) \int_G f(x)\, d\mu_G(x)
\end{equation}
and
\begin{equation}\label{eq:ModularI}
\int_G f(x^{-1})\, d\mu_G(x) = \int_G f(x) \Delta_G(x^{-1})\, d\mu_G(x)
\end{equation}
for all $f\in L^1(G,\mu_G)$ and $y\in G$. When $\Delta_G \equiv 1$, $G$ is called \emph{unimodular}. The modular function is a homomorphism with respect to multiplication on $(0,\infty)$, and it is independent of the choice of Haar measure on $G$. The subgroup $\Gamma$ also has a modular function $\Delta_\Gamma$ and a left Haar measure $\mu_\Gamma$, whose scale we also fix.

A \emph{rho function} for the pair $(G,\Gamma)$ is a continuous map $\rho \colon G \to (0,\infty)$ with the property that
\[ \rho(x\xi) = \rho(x) \frac{\Delta_\Gamma(\xi)}{\Delta_G(\xi)} \]
for all $x\in G$ and all $\xi \in \Gamma$. Such a function always exists; fix a choice once and for all, taking $\rho = 1$ if possible. There is a unique positive regular Borel measure $\mu_{G/\Gamma}$ on $G/\Gamma$ such that
\begin{equation}\label{eq:RhoMeas}
\int_G f(x)\rho(x) \, d\mu_G(x) = \int_{G/\Gamma} \int_\Gamma f(x\xi)\, d\mu_\Gamma(\xi)\, d\mu_{G/\Gamma}(x\Gamma)
\end{equation}
for all $f \in L^1(G)$. See Folland \cite[Section 2.6]{F} and Reiter and Stegeman \cite[Section 8.2]{RS}. In particular, the inner integral does not depend on the choice of coset representative for $x\Gamma$, the mapping $\xi \mapsto f(x\xi)$ belongs to $L^1(\Gamma)$ for $\mu_{G/\Gamma}$-a.e.\ $x\Gamma \in G/\Gamma$, and the function $x\Gamma \mapsto \int_\Gamma f(x\xi)\, d\mu_\Gamma(\xi)$ is measurable on $G/\Gamma$. The associated measure $\mu_{G/\Gamma}$ is strongly quasi-invariant under the action of $G$, in the sense that
\begin{equation}\label{eq:QuasInv}
\int_{G/\Gamma} f(y^{-1}x\Gamma)\, d\mu_{G/\Gamma}(x\Gamma) = \int_{G/\Gamma} f(x\Gamma) \frac{\rho(yx)}{\rho(x)}\, d\mu_{G/\Gamma}(x\Gamma)
\end{equation}
for all $f\in C_c(G/\Gamma)$ and all $y\in G$. In particular, $\mu_{G/\Gamma}$ is invariant under the action of $G$ if and only if $\rho = 1$; this can happen if and only if $\Delta_\Gamma = \left. \Delta_G \right|_\Gamma$. In that case, \eqref{eq:RhoMeas} becomes \emph{Weil's formula},
\begin{equation}\label{eq:lWeil}
\int_G f(x)\, d\mu_G(x) = \int_{G/\Gamma} \int_\Gamma f(x\xi)\, d\mu_\Gamma(\xi)\, d\mu_{G/\Gamma}(x\Gamma) \quad \text{for all $f\in L^1(G)$.}
\end{equation}
For instance, when $\Gamma$ is normal in $G$, any choice of left Haar measure on $G/\Gamma$ is invariant under the action of $G$. Therefore we can take $\rho = 1$, and by \eqref{eq:QuasInv}, $\mu_{G/\Gamma}$ is the unique left Haar measure on $G/\Gamma$ that satisfies \eqref{eq:lWeil}.

By a well-known result of Feldman and Greenleaf \cite{FG}, there is a Borel measurable function $\tau \colon G/\Gamma \to G$ with the property that $q\circ \tau = id_{G/\Gamma}$, where $q$ is the quotient mapping onto $G/\Gamma$. In effect, $\tau$ chooses a representative for each coset of $\Gamma$ in $G$, and it does it in a measurable way. We call such a function a \emph{Borel section} for $G/\Gamma$. To describe $\tau$, it suffices to give its \emph{fundamental domain} $\tau(G/\Gamma)$, since $\tau(x\Gamma)$ is the unique element of $\tau(G/\Gamma) \cap x\Gamma$. Moreover, $\tau$ is a Borel measurable function if and only if its fundamental domain is a Borel subset of $G$. As remarked in \cite{FG}, $\tau$ can be chosen such that, whenever $K\subset G/\Gamma$ is compact, $\tau(K)$ has compact closure in $G$. Fix a section $\tau$ with this property once and for all, and let $T\colon \Gamma\times G/\Gamma \to G$ be the associated bijection 
\[ T(\xi, x\Gamma) = \tau(x\Gamma)\xi. \]

\begin{prop}\label{prop:MeasIsom}
The function $T$ described above is an isomorphism of measure spaces
\[ T \colon \left(\Gamma \times G/\Gamma,\,d\mu_\Gamma \otimes d\mu_{G/\Gamma} \right) \to \left(G,\, \rho\,  d\mu_G\right). \]
\end{prop}

This result was stated in a section of notes by Folland \cite[\S2.7]{F}. However, its proof was only sketched, and an important detail was missing. The full proof relies on the remarkable lemma below. We remind the reader that a separable topological space is called \emph{Polish} if it admits a complete metric. Every second countable locally compact Hausdorff space is Polish, and the product of Polish spaces is Polish; see Kechris \cite[Theorem 5.3, Proposition 3.3]{K}. The lemma below is Theorem 14.12 of \cite{K}.

\begin{lemma}\label{lm:MeasBij}
Let $X$ and $Y$ be Polish spaces. Then every Borel measurable bijection $f\colon X \to Y$ is an isomorphism of Borel spaces. That is, $f^{-1}$ is also Borel measurable.
\end{lemma}

\noindent \emph{Proof of Proposition \ref{prop:MeasIsom}.}
For each $\xi \in \Gamma$, the function $T_\xi \colon G/\Gamma \to G$ given by $T_\xi(x\Gamma) = T(\xi, x\Gamma) = \tau(x\Gamma)\xi$ is Borel measurable; and for each $x\Gamma \in G/\Gamma$, the function $T^{x\Gamma} \colon \Gamma \to G$ given by $T^{x\Gamma}(\xi) = T(\xi, x\Gamma) = \tau(x\Gamma)\xi$ is continuous. It follows that $T$ is Borel measurable (see for instance \cite[Exercise 11.3]{K}). By Lemma \ref{lm:MeasBij} and \eqref{eq:RhoMeas}, $T$ is an isomorphism of measure spaces.
\hfill \qed

\begin{cor}\label{cor:LMeasUn}
There is a unitary $U \colon L^2(G,\mu_G) \to L^2(\Gamma \times G/\Gamma, \mu_\Gamma \otimes \mu_{G/\Gamma})$ such that
\[ Uf(\xi, x\Gamma) = \frac{f(\tau(x\Gamma) \xi)}{\sqrt{\rho(\tau(x\Gamma) \xi)}} \]
for all $f\in L^2(G)$, $\mu_\Gamma$-a.e.\ $\xi \in \Gamma$, and $\mu_{G/\Gamma}$-a.e.\ $x\Gamma \in G/\Gamma$.
\end{cor}

\begin{proof}

Let $V \colon L^2(G,\, d\mu_G) \to L^2(G,\, \rho\, d\mu_G)$ be the isometric isomorphism given by $V(f) = f/\sqrt{\rho}$. Follow $V$ by $W \colon L^2(G,\, \rho\, d\mu_G) \to L^2(\Gamma\times G/\Gamma,\, d\mu_\Gamma \otimes d\mu_{\Gamma\otimes G/\Gamma})$, $W(g) = g \circ T$. The resulting unitary is $U$.
\end{proof}

If we are willing to sacrifice invariance under the action of $G$, we can eliminate the rho function in the results above by replacing left cosets with right. Denote $q_L \colon G \to G/\Gamma$ and $q_R \colon G \to \Gamma\backslash G$ for the respective quotient maps, and $\varepsilon \colon \Gamma\backslash G \to G/\Gamma$ for the homeomorphism $\varepsilon(\Gamma x) = x^{-1}\Gamma$. If $i\colon G \to G$ is the inversion map, one easily checks that $\varepsilon \circ q_R \circ i = q_L$, as shown below.
\[ \xymatrix{
G \ar[d]^{q_L} \ar[r]^i & G \ar[d]^{q_R} \\
G/\Gamma \ar@{..>}@/^1pc/[u]^\tau & \Gamma\backslash G \ar[l]^\varepsilon
}\]
Define $\gamma \colon \Gamma\backslash G \to G$ by $\gamma = i\circ \tau \circ \varepsilon$, that is,
\begin{equation}\label{eq:RSec}
\gamma(\Gamma x) = \left[ \tau(x^{-1} \Gamma) \right]^{-1}.
\end{equation}
We claim that $\gamma$ is a Borel section for $\Gamma\backslash G$. Indeed, $\varepsilon \circ q_R \circ i \circ \tau = q_L \circ \tau = id_{G/\Gamma}$, so $q_R \circ i \circ \tau = \varepsilon^{-1}$. Hence
\[ q_R \circ \gamma = q_R \circ i \circ \tau \circ \varepsilon = \varepsilon^{-1} \circ \varepsilon = id_{\Gamma\backslash G}. \]
The claim follows once we observe that $\gamma$ is formed by composing $\tau$ with homeomorphisms on either side, so it is Borel. Moreover, it inherits the property from $\tau$, that whenever $K\subset \Gamma\backslash G$ is compact, $\gamma(K)$ has compact closure in $G$.

\begin{theorem}
There is a unique positive regular Borel measure $\mu_{\Gamma\backslash G}$ on $\Gamma\backslash G$ such that
\begin{equation}\label{eq:RWeil}
\int_G f(x)\, d\mu_G(x) = \int_{\Gamma\backslash G} \int_\Gamma f\left(\xi \gamma(\Gamma x)\right)\, d\mu_\Gamma(\xi)\, d\mu_{\Gamma\backslash G}(\Gamma x)
\end{equation}
for all $f\in L^1(G)$. In particular, $\xi \mapsto f(\xi \gamma(\Gamma x))$ belongs to $L^1(\Gamma)$ for $\mu_{\Gamma\backslash G}$-a.e.\ $\Gamma x \in \Gamma\backslash G$, and the function $\Gamma x \mapsto \int_\Gamma f(\xi \gamma(\Gamma x))\, d\mu_\Gamma(\xi)$ is measurable on $\Gamma\backslash G$. If $\varepsilon \colon \Gamma \backslash G \to G/\Gamma$ is the homeomorphism $\varepsilon(\Gamma x) = x^{-1}\Gamma$, then 
\begin{equation}\label{eq:RightMeas}
\int_{\Gamma\backslash G} f(\Gamma x)\; d\mu_{\Gamma\backslash G}(\Gamma x) = \int_{G/\Gamma} f(\varepsilon^{-1}(x\Gamma))\, \frac{1}{\rho(\tau(x\Gamma)) \Delta_G(\tau(x\Gamma))} \; d\mu_{G/\Gamma}(x\Gamma)
\end{equation}
for $f\in L^1(\Gamma\backslash G)$.
\end{theorem}

\begin{proof}
We'll first show that \eqref{eq:RightMeas} defines a measure $\mu_{\Gamma\backslash G}$ on $\Gamma\backslash G$ satisfying \eqref{eq:RWeil}. Recall that the image of a compact set in $G/\Gamma$ has compact closure in $G$. If $f\colon G \to \R$ is any continuous function, it follows that $f \circ \tau\colon G/\Gamma \to \R$ is Borel measurable and bounded on compact subsets. In particular, $x\Gamma\mapsto \frac{1}{\rho (\tau(x\Gamma)) \Delta_G(\tau(x\Gamma))}$ is a locally $\mu_{G/\Gamma}$-integrable function on $G/\Gamma$, and we can use it to define a positive regular Borel measure $d\tilde{\mu}_{G/\Gamma} = \frac{1}{(\rho\cdot \Delta_G)\circ \tau}\, d\mu_{G/\Gamma}$. Since $\varepsilon \colon \Gamma\backslash G \to G/\Gamma$ is a homeomorphism, there is a positive regular Borel measure $\mu_{\Gamma\backslash G}$ on $\Gamma\backslash G$ given by $d\mu_{\Gamma\backslash G}(\Gamma x) = d\tilde{\mu}_{G/\Gamma}(\varepsilon(\Gamma x))$. For a Borel set $E\subset \Gamma\backslash G$, this means that
\[ \int_{\Gamma\backslash G} \mathbf{1}_E(\Gamma x)\, d\mu_{\Gamma\backslash G}(\Gamma x) = \mu_{\Gamma\backslash G}(E) = \tilde{\mu}_{G/\Gamma}(\varepsilon(E)) = \int_{G/\Gamma} \mathbf{1}_{\varepsilon(E)}(x\Gamma) \frac{1}{\rho(\tau(x\Gamma))\Delta_G(\tau(x\Gamma))}\, d\mu_{G/\Gamma}(x\Gamma) \]
\[ = \int_{G/\Gamma} \mathbf{1}_E(\varepsilon^{-1}(x\Gamma)) \frac{1}{\rho(\tau(x\Gamma))\Delta_G(\tau(x\Gamma))}\, d\mu_{G/\Gamma}(x\Gamma). \]
It follows that \eqref{eq:RightMeas} holds for all $f\in L^1(\Gamma\backslash G, \mu_{\Gamma\backslash G})$.

Given $f\in L^1(G)$, use \eqref{eq:RhoMeas} to compute
\begin{align*}
\int_G f(x)\, d\mu_G(x) &= \int_G \frac{f(x^{-1})}{\rho(x)} \Delta_G(x^{-1}) \rho(x)\, d\mu_G(x) = \int_{G/\Gamma} \int_\Gamma \frac{f( (x\xi)^{-1})}{\rho(x\xi)} \Delta_G((x\xi)^{-1})\, d\mu_\Gamma(\xi)\, d\mu_{G/\Gamma}(x\Gamma) \\[7 pt]
&= \int_{G/\Gamma} \int_\Gamma \frac{f(\xi^{-1} \tau(x\Gamma)^{-1})}{\rho(\tau(x\Gamma)\xi)} \Delta_G(\xi^{-1} \tau(x\Gamma)^{-1})\, d\mu_\Gamma(\xi)\, d\mu_{G/\Gamma}(x\Gamma) \\[7 pt]
&= \int_{G/\Gamma} \int_\Gamma \frac{f(\xi^{-1} \gamma(\Gamma x^{-1})) \Delta_G(\xi)}{\rho(\tau(x\Gamma)) \Delta_\Gamma(\xi)} \Delta_G(\xi^{-1}) \Delta_G(\tau(x\Gamma)^{-1})\, d\mu_\Gamma(\xi)\, d\mu_{G/\Gamma}(x\Gamma) \\[7 pt]
&= \int_{G/\Gamma} \int_\Gamma \frac{f\left(\xi^{-1} \gamma(\varepsilon^{-1}(x \Gamma))\right)}{\rho(\tau(x\Gamma))\Delta_G(\tau(x\Gamma))} \Delta_\Gamma(\xi^{-1})\, d\mu_\Gamma(\xi)\, d\mu_{G/\Gamma}(x\Gamma) \\[7 pt]
&= \int_{G/\Gamma} \int_\Gamma f\left( \xi \gamma(\varepsilon^{-1}(x \Gamma))\right)\, d\mu_\Gamma(\xi)\, \frac{1}{\rho(\tau(x\Gamma)) \Delta_G(\tau(x\Gamma))}\, d\mu_{G/\Gamma}(x\Gamma) \\[7 pt]
&= \int_{\Gamma\backslash G} \int_\Gamma f(\xi \gamma(\Gamma x))\, d\mu_\Gamma(\xi)\, d\mu_{\Gamma\backslash G}(\Gamma x),
\end{align*}
where the penultimate equation uses the inversion formula \eqref{eq:ModularI} on $\Gamma$. Therefore $\mu_{\Gamma \backslash G}$ satisfies \eqref{eq:RWeil}.

It remains to prove uniqueness. Recall from Folland \cite[Proposition (2.48)]{F} that the periodization operator $P \colon C_c(G) \to C_c(G/\Gamma)$ given by
\[ (Pf)(x\Gamma) = \int_\Gamma f(\tau(x\Gamma)\xi)\, d\mu_\Gamma(\xi) \]
is surjective, and if $\phi \in C_c(G/\Gamma)$ is nonnegative, we can find nonnegative $f \in C_c(G)$ with $Pf = \phi$. By an argument analogous to the one given in \cite[Proposition (2.48)]{F}, there is a surjective operator $\tilde{P} \colon C_c(G) \to C_c(\Gamma\backslash G)$ given by
\[ (\tilde{P} f)(\Gamma x) = \int_\Gamma f(\xi \gamma(\Gamma x))\, d\mu_\Gamma(\xi), \]
and for $\phi \in C_c(\Gamma\backslash G)$ with $\phi \geq 0$, we can find $f \geq 0$ in $C_c(G)$ with $\tilde{P}f = \phi$; we leave it to the reader to make the necessary adjustments. Moreover, \eqref{eq:RWeil} shows that when $f,g \in C_c(G)$ are functions with $\tilde{P} f = \tilde{P} g$, $\int_G f(x)\, d\mu_G(x) = \int_G g(x)\, d\mu_G(x)$. Therefore $\tilde{P} f \mapsto \int_G f(x)\, d\mu_G(x)$ is a well-defined positive linear functional on $C_c(\Gamma\backslash G)$, and by the uniqueness in the Riesz Representation Theorem, there is only one positive regular Borel measure $\mu_{\Gamma\backslash G}$ satisfying \eqref{eq:RWeil}.
\end{proof}

A word of warning: this measure is \emph{not} usually invariant under the right action of $G$, even when an invariant measure exists, unless $G$ is unimodular. Indeed, a right invariant measure on $\Gamma\backslash G$, suitably normalized, would cause \eqref{eq:RWeil} to hold with \emph{right} Haar measure on $G$ in place of the \emph{left} Haar measure $\mu_G$. A staightforward (but tedious) computation involving \eqref{eq:RightMeas} and \eqref{eq:QuasInv} produces
\[ \int_{\Gamma\backslash G} f(\Gamma xy)\, d\mu_{\Gamma\backslash G}(\Gamma x) = \int_{\Gamma\backslash G} f(\Gamma x) \frac{ \rho\left( \gamma(\Gamma x)^{-1}\right) \Delta_G\left(\gamma(\Gamma x)^{-1} \right) }{ \rho\left( \gamma(\Gamma xy^{-1})^{-1}\right) \Delta_G\left( \gamma(\Gamma xy^{-1})^{-1} \right) } \frac{ \rho\left(yx^{-1}\right) }{ \rho\left(x^{-1}\right) }\, d\mu_{\Gamma\backslash G}(\Gamma x) \]
for all $f\in C_c(\Gamma\backslash G)$ and $y\in G$. 

\smallskip

\begin{rem}\label{rem:FundDom}
When $\Gamma$ is discrete and $\mu_\Gamma$ is counting measure, $\gamma$ identifies $(\Gamma\backslash G, \mu_{\Gamma\backslash G})$ as a measure space with $(\gamma(\Gamma\backslash G), \mu_G)$, but when $\Gamma$ is not discrete, $\mu_G(\gamma(\Gamma\backslash G))=0$. To see this, let $E\subset \gamma(\Gamma\backslash G)$ be a Borel set with $\mu_G(E) < \infty$, and use \eqref{eq:RWeil} to compute
\[ \mu_G(E) = \int_G \mathbf{1}_E(x)\, d\mu_G(x) = \int_{\Gamma\backslash G} \int_\Gamma \mathbf{1}_E(\xi \gamma(\Gamma x))\, d\mu_\Gamma(\xi)\, d\mu_{\Gamma\backslash G}(\Gamma x) \]
\[ = \int_{\Gamma\backslash G} \mu_\Gamma(\{1\})\cdot \mathbf{1}_E(\gamma(\Gamma x))\, d\mu_{\Gamma\backslash G}(\Gamma x) = \mu_\Gamma(\{1\}) \int_{\Gamma\backslash G} \mathbf{1}_{\gamma^{-1}(E)}(\Gamma x)\, d\mu_{\Gamma\backslash G}(\Gamma x) \]
\[ = \mu_\Gamma(\{1\})\cdot \mu_{\Gamma\backslash G}(\gamma^{-1}(E)). \]
Lemma \ref{lm:MeasBij} shows that $\gamma$ preserves the Borel $\sigma$-algebra, and the claim follows.
\end{rem}

\begin{theorem} \label{thm:RMeasIsom}
The mapping
\[ T \colon \left( \Gamma \times \Gamma\backslash G,\, d\mu_\Gamma \otimes d\mu_{\Gamma\backslash G} \right) \to \left( G,\, d\mu_G \right) \]
given by $T(\xi, \Gamma x) = \xi \gamma(\Gamma x)$
is a measure space isomorphism.
\end{theorem}

\begin{proof}
It follows from Lemma \ref{lm:MeasBij} and \eqref{eq:RWeil} just as Proposition \ref{prop:MeasIsom} did.
\end{proof}

\begin{cor}\label{cor:RMeasUn}
There is a unitary map $U\colon L^2(G,\mu_G) \to L^2(\Gamma\times \Gamma\backslash G, \mu_\Gamma \otimes \mu_{\Gamma\backslash G} )$ such that
\[ Uf(\xi, \Gamma x) = f(\xi \gamma(\Gamma x) ) \]
for all $f\in L^2(G)$, $\mu_\Gamma$-a.e.\ $\xi \in \Gamma$, and $\mu_{\Gamma\backslash G}$-a.e.\ $\Gamma x \in \Gamma\backslash G$.
\end{cor}

\begin{rem}\label{rem:NormMeas}
When $\Gamma$ is a \emph{normal} subgroup of $G$, we have defined two measures on the coinciding quotient spaces $G/\Gamma = \Gamma\backslash G$, namely $\mu_{G/\Gamma}$ and $\mu_{\Gamma\backslash G}$. In the most general setting, these measures need not be equal, but they are related in a way that we now describe. For arbitrary $f\in C_c(G/\Gamma)$, we compute
\[ \int_{\Gamma\backslash G} f(\Gamma x)\, d\mu_{\Gamma\backslash G}(\Gamma x) =  \int_{G/\Gamma} f(\Gamma x^{-1}) \Delta_G(\tau(x\Gamma)^{-1}) d\mu_{G/\Gamma}(x\Gamma) \]
\[ = \int_{G/\Gamma} f(\Gamma x^{-1}) \Delta_G(\gamma(\Gamma x^{-1})) d\mu_{G/\Gamma}(x\Gamma). \]
Since $\Gamma$ is normal in $G$, $\mu_{G/\Gamma}$ is a left Haar measure on $G/\Gamma$. Denoting $\Delta_{G/\Gamma}$ for the modular function on $G/\Gamma$, and identifying $\Gamma x^{-1}$ with $x^{-1}\Gamma$, we compute
\[ \int_{\Gamma\backslash G} f(\Gamma x)\, d\mu_{\Gamma\backslash G}(\Gamma x) = \int_{G/\Gamma} f(x\Gamma) \Delta_G(\gamma(\Gamma x)) \Delta_{G/\Gamma}(x^{-1} \Gamma)\, d\mu_{G/\Gamma}(x\Gamma), \]
by \eqref{eq:ModularI}. Thus,
\begin{equation}\label{eq:LRRadNik}
d\mu_{\Gamma\backslash G} = \frac{ \Delta_G\circ \gamma }{ \Delta_{G/\Gamma} }\, d\mu_{G/\Gamma}.
\end{equation} 

There is another way to compute the Radon-Nikodym derivative that is sometimes useful. For each $x\in G$, there is a unique number $\delta(x)>0$ such that
\[ \int_\Gamma f(x \xi x^{-1})\, d\mu_\Gamma(\xi) = \delta(x) \int_\Gamma f(\xi)\, d\mu_\Gamma(\xi) \]
for all $f\in L^1(\Gamma)$. In fact, $\delta(x) = \frac{\Delta_G(x) }{ \Delta_{G/\Gamma}(x\Gamma) }$. (See Nachbin \cite[Chapter II, Propositions 16 and 22]{N}.) Thus we can take any function $f\in L^1(\Gamma)$ with nonzero integral that we like, and compute
\[ \frac{\Delta_G(\gamma(\Gamma x)) }{ \Delta_{G/\Gamma}(x\Gamma) } = \delta(\gamma(x\Gamma)) = \int_\Gamma f\left(\gamma(\Gamma x) \xi \gamma(\Gamma x)^{-1}\right)\, d\mu_\Gamma(\xi) \cdot \left( \int_\Gamma f(\xi)\, d\mu_\Gamma(\xi) \right)^{-1}. \]
For instance, in the case where $\Gamma$ is compact and normal, we can take $f=1$ in the formula above to see that $\frac{\Delta_G \circ \gamma}{\Delta_{G/\Gamma}} = 1$, and therefore $\mu_{\Gamma\backslash G} = \mu_{G/\Gamma}$. Likewise, $\mu_{\Gamma\backslash G} = \mu_{G/\Gamma}$ when $\Gamma$ is a closed subgroup in the center of $G$.
\end{rem}

\begin{example}\label{examp:ax+b}
Let $G$ be the affine group on $\R$ consisting of transformations $x \mapsto ax+b$ with $a>0$. As a topological space, we identify $G$ with $(0,\infty) \times \R$; its group laws are then given by
\[ (a,b) \cdot (c,d) = (ac, b+ ad) \quad \text{and} \quad (a,b)^{-1} = (1/a, -b/a). \]
The modular function is $\Delta_G(a,b) = 1/a$, and a left Haar measure is given by
\[ d\mu_G(a,b) = \frac{da\, db}{a^2}. \]

Let $H$ be the normal subgroup
\[ H = \{ (1,b) \in G : b \in \R\}, \]
and let
\[ K = \{ (a,0) \in G : a > 0 \}. \]
Then $H \cong (\R,+)$, $K \cong (\R_+, \times)$, and $G = H \rtimes K$. In particular, $G/H \cong K$. We choose Borel sections $\tau_H \colon G/H \to G$ and $\tau_K \colon G/K \to G$ given by
\[ \tau_H((a,b)H) = (a,0) \quad \text{and} \quad \tau_K((a,b)K) = (1,b). \]
The associated sections $\gamma_H \colon H\backslash G \to G$ and $\gamma_K \colon K\backslash G \to G$ then have formulae
\[ \gamma_H(H(a,b)) = (a,0) \quad \text{and} \quad  \gamma_K(K(a,b)) = (1,b/a). \]

For Haar measures on $H$ and $K$, we choose $d\mu_H(1,b) = db$ and $d\mu_K(a,0) = da/a$, respectively.  Let us compute the measures on the respective quotients. Since $H$ is a normal subgroup, $\mu_{G/H}$ is the unique choice of left Haar measure on $G/H \cong K \cong (\R_+, \times)$ satisfying \eqref{eq:lWeil}; an easy computation shows that
\[ d\mu_{G/H}((a,b)H) = d\mu_K(a,0) = \frac{da}{a}. \]
Then by \eqref{eq:LRRadNik},
\[ d\mu_{H\backslash G}(H(a,b)) = \frac{\Delta_G(\gamma_H(H(a,b)))}{\Delta_{G/H}((a,b)H)}\, d\mu_{G/H}(a,b) = \Delta_G(a,0)\, \frac{da}{a} = \frac{da}{a^2}. \]

On the other hand, $K$ is not a normal subgroup, and since $\Delta_G(a,0) \neq \Delta_K(a,0)$ there is no invariant measure on $G/K$. However, a rho function is given by $\rho_K(a,b) = a$, and the associated quasi-invariant measure on $G/K \cong \R$ is
\[ d\mu_{G/K}((a,b)K) = db, \]
as the reader can easily verify.

The interested reader can now compute $\mu_{K\backslash G}$ using \eqref{eq:RightMeas}. We proceed straight to the punchline. Since $G/K$ is homeomorphic with $\R$ via $(1,b)K \mapsto b$, and since $(1,b)^{-1} = (1,-b)$, composing with $\varepsilon\colon K\backslash G \to G/K$ on the left and with $b \mapsto -b$ on the right shows that $K\backslash G \cong \R$ via $K(1,b) \mapsto b$. In particular, there is a positive regular Borel measure $\mu_{K\backslash G}$ on $K\backslash G$ given by
\[ d\mu_{K\backslash G}(K(1,b)) = db. \]
For any $f\in C_c(G)$, we compute
\[ \int_{K\backslash G} \int_K f\Bigl((a,0)\cdot\gamma(K(1,b))\Bigr)\, d\mu_K(a,0)\, d\mu_{K\backslash G}(K(1,b)) = \int_{K\backslash G} \int_K f\bigl((a,0)\cdot(1,b)\bigr)\,d\mu_{K}(a,0)\, d\mu_{K\backslash G}( K(1,b) ) \]
\[ = \int_\R \int_{\R_+} f(a,ab)\, \frac{da}{a}\, db = \int_\R \int_{\R_+} f(a,b) \frac{da\,db}{a^2} = \int_G f(a,b)\, d\mu_G(a,b). \]
Therefore $\mu_{K\backslash G}$ satisfies \eqref{eq:RWeil}. In other words, $K(1,b) \mapsto b$ identifies $K\backslash G$ with $\R$ as both a topological space and a measure space.
\end{example}

\medskip

%ZAK TRANSFORM AND FIBERIZATION=========================================
\section{The Zak transform and fiberization}\label{sec:Zak}

Let $G$ be a second countable locally compact group with a closed \emph{abelian} subgroup $H$, with notation as in the last section. In this section we develop a generalized version of the Zak transform for the pair $(G,H)$. 

We begin with a short reminder of terminology on locally compact abelian (LCA) groups. Let $\mathcal{G}$ be an LCA group, and let $\mu_{\mathcal{G}}$ be a Haar measure on $\mathcal{G}$. We denote $\hat{\mathcal{G}}$ for the dual group of $\mathcal{G}$, which consists of continuous homomorphisms $\alpha \colon \mathcal{G} \to \T$ under pointwise multiplication, with the topology of uniform convergence on compact sets. The \emph{Fourier transform} of $f\in L^1(\mathcal{G})$ is the function $\hat{f} \in C_0(\hat{\mathcal{G}})$ given by
\[ \hat{f}(\alpha) = \int_{\mathcal{G}} f(x)\, \alpha(x^{-1})\, d\mu_{\mathcal{G}}(x). \]
For any choice of Haar measure on $\hat{\mathcal{G}}$, the Fourier transform maps $L^1(\mathcal{G}) \cap L^2(\mathcal{G})$ onto a dense subspace of $L^2(\hat{\mathcal{G}})$, and for a unique choice $\mu_{\hat{\mathcal{G}}}$ this map is an isometry. That choice of $\mu_{\hat{\mathcal{G}}}$ is called \emph{dual} to $\mu_{\mathcal{G}}$; it is always the measure we have in mind. With dual Haar measure on $\hat{\mathcal{G}}$, the Fourier transform extends uniquely to a unitary map $\F_{\mathcal{G}} \colon L^2(\mathcal{G}) \to L^2(\hat{\mathcal{G}})$. We also call $\F_{\mathcal{G}}$ the Fourier transform, and we also denote $\hat{f} = \F_{\mathcal{G}} f$ for $f\in L^2(\mathcal{G})$. When $g \in L^2(\hat{\G})$, we denote $\check{g} = \F_{\G}^{-1} g$. For any $f\in L^1(\G) + L^2(\G)$, the Fourier transform satisfies the following intertwining relation:
\[ ( L_y f )^\wedge(\alpha) = \alpha(y^{-1}) \hat{f}(\alpha) \quad \text{for }y \in \mathcal{G},\, \alpha \in \hat{\mathcal{G}}. \]

The dual of $\hat{\G}$ can be identified with $\G$ as follows. Each $x\in \G$ defines a character $X_x$ on $\hat{\G}$ given by $X_x(\alpha) = \alpha(x)$, and the mapping $x \mapsto X_x$ is a topological group isomorphism of $\G$ with $\hat{\hat{\G}}$. This is called Pontryagin Duality. When $\hat{\hat{\G}}$ is identified with $\G$ in this way, $\mu_\G$ gives the measure dual to $\mu_{\hat{\G}}$.

If $\Gamma \subset \G$ is a closed subgroup, its \emph{annihilator} in $\hat{\G}$ is the closed subgroup
\[ \Gamma^* = \{ \kappa \in \hat{\G} : \kappa(\xi) = 1 \text{ for all }\xi \in \Gamma\}. \]
The subgroups $\Gamma\subset \G$, $\Gamma^* \subset \hat{\G}$, and their quotients are all canonically related through duality. First, each $\kappa \in \Gamma^*$ defines a character $\hat{\kappa} \in (\G/\Gamma)^\wedge$ by the formula
\begin{equation}\label{eq:KappaHat}
\hat{\kappa}(x\Gamma) = \kappa(x),
\end{equation}
and the mapping $\kappa \mapsto \hat{\kappa}$ identifies $\Gamma^*$ with $(\G/\Gamma)^\wedge$ as topological groups. Likewise, $\hat{\G}/\Gamma^*$ identifies with $\hat{\Gamma}$ through the mapping $\omega \Gamma^* \mapsto \left. \omega \right|_\Gamma$. Moreover, the dual measures on $\hat{\G}$, $\Gamma^* \cong (\G/\Gamma)^\wedge$, and $\hat{\G}/\Gamma^* \cong \hat{\Gamma}$ satisfy Weil's formula \eqref{eq:lWeil}.

Given $f\colon G \to \C$ and a coset $Hx \in H\backslash G$, we will write $f^{Hx} \colon H \to \C$ for the function
\[ f^{Hx}(\xi) = f(\xi \gamma(Hx)), \]
where $\gamma \colon H\backslash G \to G$ is the Borel section from \eqref{eq:RSec}. Given $\varphi \colon \hat{H} \to L^2(H\backslash G)$, we define functions $\varphi_{Hx} \colon \hat{H} \to \C$ for a.e.\ $Hx \in H\backslash G$ with the formula
\[ \varphi_{Hx}(\alpha) = \varphi(\alpha)(Hx). \]

\smallskip

\begin{theorem}\label{thm:Zak}
There is a unitary transformation $Z \colon L^2(G) \to L^2(\hat{H}; L^2(H\backslash G))$ given by
\begin{equation}\label{eq:Zak}
(Zf)(\alpha)(Hx) = \widehat{f^{Hx}}(\alpha) \qquad \text{for all $f\in L^2(G)$, a.e.\ $\alpha \in \hat{H}$, and a.e.\ $Hx \in H\backslash G$.}
\end{equation}
Its inverse is given by
\begin{equation}\label{eq:ZakInv}
(Z^{-1} \varphi)(\xi \gamma(Hx)) = \widecheck{\varphi_{Hx}}(\xi) \quad \text{for all $\varphi \in L^2(\hat{H}; L^2(H\backslash G))$, a.e.\ $\xi \in H$, and a.e.\ $Hx \in H \backslash G$.}
\end{equation}
When $f\in L^2(G)$ and $\xi \in H$, $Z$ satisfies the relation
\begin{equation}\label{eq:ZakTrans}
(Z L_\xi f)(\alpha) = \alpha(\xi^{-1}) \cdot (Zf)(\alpha)
\end{equation}
for a.e.\ $\alpha \in \hat{H}$.
\end{theorem}

\begin{proof}
Construct a sequence of unitaries
\[ L^2(G) \overset{U_1}{\to} L^2(H\times H\backslash G) \overset{U_2}{\to} L^2(H\backslash G ; L^2(H)) \overset{U_3}{\to} L^2(H\backslash G; L^2(\hat{H})) \overset{U_4}{\to} L^2(\hat{H}; L^2(H\backslash G)), \]
where $U_1$ is the isomorphism from Corollary \ref{cor:RMeasUn}, $U_3$ is the unitary given by
\[ (U_3 \varphi)(Hx) = \widehat{\varphi(Hx)}, \]
and all others are the natural isomorphisms. Let $Z=U_4 U_3 U_2 U_1$. Then
\[ (U_2 U_1 f)(Hx)(\xi) = (U_1 f)(\xi, Hx) = f(\xi \gamma(Hx)) = f^{Hx}(\xi), \]
and
\[ (Zf)(\alpha)(Hx) = (U_4 U_3 U_2 U_1 f)(\alpha)(Hx) = (U_3 U_2 U_1 f)(Hx)(\alpha) = [(U_2 U_1 f)(Hx)]^{\wedge}(\alpha) = \widehat{f^{Hx}}(\alpha). \]
This proves \eqref{eq:Zak}. A similar computation verifies \eqref{eq:ZakInv}. Moreover, for every $f\colon G \to \C$ and every $\xi \in H$,
\[(L_\xi f)^{Hx}(\eta) = (L_\xi f)(\eta \gamma(Hx)) = f(\xi^{-1} \eta \gamma(Hx)) = f^{Hx}(\xi^{-1} \eta) = L_\xi(f^{Hx})(\eta). \]
Hence, for $f\in L^2(G)$ and $\xi \in H$,
\[ (Z L_\xi f)(\alpha)(Hx) = [ (L_\xi f)^{Hx} ]^{\wedge}(\alpha) = [ L_\xi (f^{Hx}) ]^{\wedge}(\alpha) = \alpha(\xi^{-1}) \widehat{f^{Hx}}(\alpha) = \alpha(\xi^{-1})\cdot (Zf)(\alpha)(Hx), \]
as in \eqref{eq:ZakTrans}.
\end{proof}

We call $Z$ the \emph{Zak transform}, for reasons that will soon be obvious. Whenever we find it useful, we will freely interpret $Z$ as the unitary $\tilde{Z} \colon L^2(G) \to L^2(\hat{H} \times H\backslash G)$ given by
\[ (\tilde{Z}f)(\alpha,Hx) = (Zf)(\alpha)(Hx) = \widehat{f^{Hx}}(\alpha). \]
We emphasize that both $Z$ \emph{and} the measure used to construct $L^2(H\backslash G)$ depend on the choice of Borel section $\gamma$. Our construction of the Zak transform generalizes the definition given by Weil in \cite[pp.\ 164--165]{W2} to the case where $G$ is nonabelian; see Example \ref{examp:Zaks}(vi) below. For more on the history of the Zak transform, we refer the reader to \cite{HSWW2}.

\begin{example}\label{examp:Zaks}
We now compute $Z$ in a wide variety of concrete settings.

(i) $\Z \subset \R$. To justify our usage of ``Zak transform'', we first compute $Z$ for the subgroup $\Z \subset \R$. Take Lebesgue measure for $\mu_\R$ and counting measure for $\mu_\Z$. We use the fundamental domain $[0,1)$. Since $\Z$ is discrete, the associated section $\gamma$ identifies $(\Z \backslash \R, \mu_{\Z \backslash \R})$ with the interval $[0,1)$ under Lebesgue measure, as explained in Remark \ref{rem:FundDom}. From this perspective, $f^{t+\Z}(k) = f(t+k)$ for $f\colon \R\to \C$, $t \in [0,1)$, and $k \in \Z$. When $\hat{\Z}$ is identified with $\T$, the Zak transform becomes the map $Z \colon L^2(\R) \to L^2(\T; L^2([0,1)))$ which for $f\in L^1(\R) \cap L^2(\R)$ is given by
\[ (Zf)(z)(t) = \widehat{f^{t+\Z}}(z) = \sum_{k\in \Z} f^{t+\Z}(k) z^{-k} = \sum_{k \in \Z} f(t+k) z^{-k}. \]
If we further identify $\T$ with the interval $[0,1)$ under Lebesgue measure, $Z$ can be thought of as the map $\tilde{Z} \colon L^2(\R) \to L^2([0,1) \times [0,1))$ given by
\begin{equation}\label{eq:ClasZak}
(\tilde{Z} f)(s,t) = \sum_{k\in \Z} f(t+k) e^{-2\pi i k s}
\end{equation}
for $f \in L^1(\R) \cap L^2(\R)$ and $s,t \in [0,1)$. This is exactly the classical Zak transform.

\bigskip

\noindent (ii) $\Z^m \subset \R^n$. More generally, let $m$ and $n$ be positive integers with $m \leq n$, and think of $\Z^m$ as the subgroup of $\R^n$ consisting of vectors with integers in the first $m$ entries and zeros in the last $n-m$. A fundamental domain is given by $[0,1)^m \times \R^{n-m}$, and since $\Z^m$ is discrete, the associated section $\gamma \colon \Z^m \backslash \R^n \to \R^n$ identifies the measure space $(\Z^m \backslash \R^n, \mu_{\Z^m \backslash \R^n})$ with $[0,1)^m \times \R^{n-m}$ under Lebesgue measure. Identifying $\widehat{\Z^m}$ with $[0,1)^m \subset \R^m$ as above, we can think of the Zak transform as a unitary
\[ \tilde{Z} \colon L^2(\R^n) \to L^2([0,1)^m \times [0,1)^m \times \R^{n-m}) \]
which for $f\in L^1(\R^n) \cap L^2(\R^n)$ is given by
\begin{equation}\label{eq:ZmRnZak}
(\tilde{Z}f)(s,t,x) = \sum_{k \in \Z^m} f(t + k, x) e^{-2\pi i k\cdot s}.
\end{equation}

\bigskip

\noindent (iii) $\R^m \subset \R^n$. Let $m$ and $n$ be positive integers with $m \leq n$, and consider $\R^m$ as the subgroup of $\R^n$ consisting of vectors with zeros in the last $n-m$ entries. Then $\R^m \backslash \R^n \cong \R^{n-m}$ with Lebesgue measure, by \eqref{eq:RhoMeas} and Remark \ref{rem:NormMeas}. Our section $\gamma \colon \R^{n-m} \to \R^n$ will be given by
\[ \gamma(x_{m+1},\dotsc, x_n) = (0,\dotsc,0,x_{m+1},\dotsc,x_n). \]
Identifying $\hat{\R^m}$ with $\R^m$ in the usual way, we can view the Zak transform for $\R^m \subset \R^n$ as a unitary
\[ Z \colon L^2(\R^n) \to L^2(\R^{m}; L^2(\R^{n-m})) \]
which for $f\in L^1(\R^n) \cap L^2(\R^n)$ is given by
\begin{equation}\label{eq:RmRnZak}
(Zf)(\xi)(y) = \int_{\R^m} f(x,y) e^{-2\pi i \xi \cdot x}\; dx.
\end{equation}

\bigskip

\noindent (iv) $\Z_p \subset \Q_p$. Let $p$ be a prime number, and let $\Q_p$ be the locally compact field of $p$-adic numbers
\[ x = \sum_{j=m}^\infty c_j p^j \]
for $m \in \Z$ and $c_j \in \{0,1,\dotsc, p-1\}$. The topology on $\Q_p$ is given by the $p$-adic norm $|\cdot|_p$; for $x$ as above with $c_m \neq 0$, $|x|_p = p^{-m}$. Any two elements of $\Q_p$ can be added or multiplied in the obvious way, and under these operations $\Q_p$ is a locally compact field. Consider $\Q_p$ as an LCA group under addition, and let $\Z_p$ be the compact open subgroup of $p$-adic integers
\[ \Z_p = \{ x \in \Q_p : |x|_p \leq 1\} = \left\{ \sum_{j=0}^\infty c_j p^j : c_j \in \{0,1,\dotsc, p-1\} \right\}. \]
A fundamental domain for $\Z_p$ is
\[ \Omega = \left\{ \sum_{j=m}^{-1} c_j p^j : m \in \Z_{<0},\, c_j \in \{0,1,\dotsc, p-1\} \right\}. \]
Since $\Z_p$ is open in $\Q_p$, the quotient $\Z_p \backslash \Q_p$ is discrete, and the section $\gamma \colon \Z_p \backslash \Q_p \to \Q_p$ associated with $\Omega$ is automatically Borel. Moreover, the image of a compact set automatically has compact closure, as required.

Identify $\hat{\Q}_p$ with $\Q_p$ as follows. For $x = \sum_{j=m}^\infty c_j p^j \in \Q_p$, we abbreviate
\[ e^{\pm2\pi i x} = \exp\left( \pm2\pi i \sum_{j=m}^{-1} c_j p^j \right). \]
Each $y \in \Q_p$ then defines a character $\omega_y \in \hat{\Q}_p$ by the formula $\omega_y(x) = e^{2\pi i y x}$, and the mapping $y \mapsto \omega_y$ is a topological group isomorphism of $\Q_p$ with $\hat{\Q}_p$. Moreover,
\[ \Z_p^* = \{ \omega_\xi : \xi \in \Z_p \}. \]
Hence $\hat{\Z}_p \cong \hat{\Q}_p / \Z_p^*$ is the discrete group of characters $\left.\omega_y\right|_{\Z_p}$ for $y \in \Omega$. 

When Haar measures are normalized so that $\mu_{\Q_p}(\Z_p) = \mu_{\Z_p}(\Z_p) = 1$, the dual measure on $\hat{\Z}_p$ is counting measure. Counting measure on $\Z_p \backslash \Q_p$ also causes \eqref{eq:RWeil} to hold. Identifying both $\hat{\Z}_p$ and $\Z_p \backslash \Q_p$ with $\Omega$ makes the Zak transform a unitary
\[ \tilde{Z} \colon L^2(\Q_p) \to l^2(\Omega \times \Omega) \]
which for $f\in L^1(\Q_p) \cap L^2(\Q_p)$ is given by
\begin{equation}\label{eq:ZpQpZak}
(\tilde{Z}f)(x,y) = \int_{\Z_p} f(y+ \xi)e^{-2\pi i x \xi}\, d\mu_{\Z_p}(\xi).
\end{equation}

\bigskip

\noindent (v) $\R \subset ax+b$. Let $G$ be the $ax+b$ group described in Example \ref{examp:ax+b}. For the normal subgroup
\[ H = \{(1,b) \in G : b \in \R\} \cong (\R,+), \]
we identify $\hat{H}$ with $\R$ in the usual way. When $L^2(H\backslash G)$ is identified with $L^2(\R_+, dx/x^2)$ via $x \mapsto H(x,0)$, the Zak transform becomes a unitary $Z_H \colon L^2(G) \to L^2(\R; L^2(\R_+,dx/x^2))$ which for $f\in L^1(G)\cap L^2(G)$ is given by
\[ (Z_H f)(\xi)(a) = \int_\R f(a,b) e^{-2\pi i \xi b}\, db. \]
On the other hand, the subgroup
\[ K = \{ (a,0) \in G : a > 0\} \]
is isomorphic with $(\R_+, \times)$, and its dual can be identified with $(\R,+)$ under the pairing
\[ \hat{\xi}(a,0) = a^{2\pi i \xi}. \]
For $(a,0) \in K$ and $K(1,b) \in K\backslash G$, we have
\[ f^{K(1,b)}(a,0) = f\bigl((a,0)\cdot \gamma_K(K(1,b))\bigr) = f\bigl((a,0)\cdot (1,b)\bigr) = f(a,ab). \]
Identifying $L^2(K\backslash G)$ with $L^2(\R)$ as in Example \ref{examp:ax+b}, 
the Zak transform becomes a unitary $Z_K \colon L^2(G) \to L^2(\R\times \R)$ which for $f\in L^1(G)\cap L^2(G)$ is given by
\[ (Z_K f)(\xi,b) = \int_0^\infty \frac{f(a,ab) a^{- 2\pi i \xi}}{a}\, da. \]

\bigskip

\noindent (vi) Let $G$ be any second countable LCA group with closed subgroup $H$. For $f\in C_c(G)$, $\alpha \in \hat{H}$, and $Hx \in H\backslash G$, we compute
\[ (\tilde{Z}f)(\alpha,Hx) = \widehat{f^{Hx}}(\alpha) = \int_H f(\xi \gamma(Hx)) \alpha(\xi^{-1})\, d\mu_H(\xi). \]
More generally, for $x \in G$ the function $\xi \mapsto f(\xi x)$ belongs to $C_c(H)$, so we can define a function $\tilde{\tilde{Z}} f \colon G \times \hat{G} \to \C$ by the formula
\begin{equation}\label{eq:AbelZak}
(\tilde{\tilde{Z}} f)(\omega,x) = \int_H f(\xi x) \omega(\xi^{-1})\, d\mu_H(\xi).
\end{equation}
This matches the definition of Zak transform given by Weil in \cite{W2} (although he didn't call it that, of course).
\end{example}

For the remainder of this section, we will assume that $G$ is abelian. Each $\omega \in \hat{G}$ then acts unitarily on $L^2(G)$ via the modulation $(M_\omega f)(x) = \omega(x) f(x)$. The Zak transform behaves well under modulations by $H^*$ and translations by $H$. When $f\in C_c(G)$ and $\kappa \in H^*$,
\[ (\tilde{Z} M_\kappa f)(\alpha, Hx) = \int_H (M_\kappa f)(\xi \gamma(Hx)) \alpha(\xi^{-1})\, d\mu_H(\xi) = \int_H \kappa(\xi \gamma(Hx)) f(\xi \gamma(Hx)) \alpha(\xi^{-1})\, d\mu_H(\xi) \]
\[ = \int_H \kappa(\gamma(Hx)) f(\xi \gamma(Hx)) \alpha(\xi^{-1})\, d\mu_H(\xi) = \kappa(\gamma(Hx))\cdot (\tilde{Z} f)(\alpha, Hx).\]
Since $\gamma(Hx) = \xi x$ for some $\xi \in H$, and since $\kappa(\xi)=1$, we can write $\kappa(x)$ in place of $\kappa(\gamma(Hx))$ in the last expression above. Extending by continuity and combining with \eqref{eq:ZakTrans}, we find that
\begin{equation}\label{eq:ZakTransMod}
(\tilde{Z} L_\xi M_\kappa f)(\alpha, Hx) = \alpha(\xi^{-1}) \kappa(x) \cdot (\tilde{Z} f)(\alpha, Hx)
\end{equation}
for all $f\in L^2(G)$, $\xi \in H$, and $\kappa \in H^*$.

\medskip

In the abelian setting, the Zak transform has a sibling, which we now introduce. Whenever we work in this setting we will use a fixed Borel section $\beta \colon \hat{G} / H^* \to \hat{G}$ that sends compact sets to pre-compact sets. We remind the reader that $G/H = H\backslash G$ and $\hat{G} / H^* = H^* \backslash \hat{G}$ as measure spaces; see the final line of Remark \ref{rem:NormMeas}.

\begin{prop}\label{prop:Fib}
In addition to the standing hypotheses, suppose that $G$ is abelian. There is a unitary map
\[ \mathcal{T} \colon L^2(G) \to L^2(\hat{G}/H^*; L^2(H^*)) \]
given by
\begin{equation}\label{eq:Fib}
(\mathcal{T}f)(\omega H^*)(\kappa) = \hat{f}(\beta(\omega H^*) \kappa).
\end{equation}
Moreover, for any $\xi \in H$,
\begin{equation}\label{eq:FibTrans}
(\mathcal{T} L_\xi f)(\omega H^*) = \omega(\xi^{-1})\cdot (\mathcal{T}f)(\omega H^*).
\end{equation}
\end{prop}

\begin{proof}

Follow the Fourier transform $L^2(G) \to L^2(\hat{G})$ by the unitary $L^2(\hat{G}) \to L^2(H^* \times \hat{G}/H^*)$ from Corollary \ref{cor:LMeasUn}. When $L^2(H^* \times \hat{G}/H^*)$ is identified with $L^2(\hat{G}/H^*; L^2(H^*))$, the composition $\mathcal{T}$ is given by \eqref{eq:Fib}.

If $f\in L^2(G)$ and $\xi \in H$, we compute
\[ (\mathcal{T} L_\xi f)(\omega H^*)(\kappa)= (L_\xi f)^\wedge(\beta(\omega H^*) \kappa) = \beta(\omega H^*)(\xi^{-1}) \kappa(\xi^{-1}) \hat{f}(\beta(\omega H^*) \kappa) = \omega(\xi^{-1})\cdot (\mathcal{T} f)(\omega H^*)(\kappa), \]
since $\beta(\omega H^*) = \omega \chi$ for some $\chi \in H^*$, and $\chi(\xi^{-1}) = \kappa(\xi^{-1}) = 1$. This proves \eqref{eq:FibTrans}.
\end{proof}

We call $\mathcal{T}$ the \emph{fiberization} map. In the special case where $H$ is discrete and $G/H$ is compact, Proposition \ref{prop:Fib} was proved separately by Kamyabi Gol and Raisi Tousi \cite[Proposition 2.1]{KR} and Cabrelli and Paternostro \cite[Proposition 3.3]{CP}. To the author's knowledge, every existing classification of $H$-TI uses some version of fiberization. 

The Zak transform is closely related to fiberization in the abelian setting, and indeed $\mathcal{T}$ can be obtained from $Z$ through a modulation in $L^2(\hat{H}; L^2(H\backslash G))$ and the Fourier transform on $H\backslash G$, as we now show. With the isomorphisms $\hat{H} \cong \hat{G}/H^*$ and $H^* \cong (G/H)^\wedge$ in mind, define a modulation $M\colon L^2(\hat{H}; L^2(H\backslash G)) \to L^2(\hat{H}; L^2(H\backslash G))$ by the formula
\begin{equation}\label{eq:ZakFibMod}
(M \varphi)(\left. \omega \right|_H)(Hx) = \beta(\omega H^*)(\gamma(Hx)^{-1})\cdot \varphi(\left. \omega \right|_H)(Hx)
\end{equation}
for $\varphi \in L^2(\hat{H};L^2(H\backslash G))$, $\omega \in \hat{G}$, and $Hx \in H\backslash G$. We claim that
\begin{equation}\label{eq:ZakVsFib}
(\mathcal{T}f)(\omega H^*)(\kappa) = [ (M Zf)(\left.\omega\right|_H) ]^\wedge(\hat{\kappa})
\end{equation}
for any $f\in L^2(G)$, where the Fourier transform on the right is taken over $H\backslash G$. Indeed, for any $f\in C_c(G)$, we compute
\[ (\mathcal{T} f)(\omega H^*)(\kappa) = \hat{f}(\beta(\omega H^*) \kappa) = \int_G f(x) \beta(\omega H^*)(x^{-1}) \kappa(x^{-1})\, d\mu_G(x) \]
\[ = \int_{H\backslash G} \int_H f(\xi \gamma(Hx))\cdot \beta(\omega H^*)(\gamma(Hx)^{-1} \xi^{-1})\cdot \kappa( \gamma(Hx)^{-1} \xi^{-1})\, d\mu_H(\xi)\, d\mu_{H\backslash G}(Hx) \]
\[ = \int_{H\backslash G} \beta(\omega H^*)(\gamma(Hx)^{-1})\cdot \kappa(\gamma(Hx)^{-1}) \int_H f(\xi \gamma(Hx))\cdot \beta(\omega H^*)(\xi^{-1})\cdot \kappa(\xi^{-1})\, d\mu_H(\xi)\, d\mu_{H\backslash G}(Hx). \]
This is messy, but it cleans up nicely. First, $\kappa(\gamma(Hx)^{-1}) = \kappa(x^{-1})$, since $\kappa(\eta) = 1$ for any $\eta \in H$. Likewise, $\beta(\omega H^*)(\xi^{-1}) = \omega(\xi^{-1})$, since any element of $H^*$ annihilates $\xi^{-1}$. We also have $\kappa(\xi^{-1}) = 1$, and we can abbreviate $f(\xi \gamma(Hx)) = f^{Hx}(\xi)$. With all that in mind, our last equation reads
\[ (\mathcal{T} f)(\omega H^*)(\kappa) = \int_{H\backslash G} \beta(\omega H^*)(\gamma(Hx)^{-1})\cdot \kappa(x^{-1}) \int_H f^{Hx}(\xi) \omega(\xi^{-1})\, d\mu_H(\xi)\, d\mu_{H\backslash G}(Hx) \]
\[ = \int_{H\backslash G} \beta(\omega H^*)(\gamma(Hx)^{-1})\cdot \widehat{f^{Hx}}(\left. \omega \right|_H)\cdot \kappa(x^{-1})\, d\mu_{H\backslash G}(Hx) \]
\[ = \int_{H\backslash G} (M Z f)(\left. \omega \right|_H)(Hx)\cdot \hat{\kappa}(Hx^{-1})\, d\mu_{H\backslash G}(Hx) = [ (M Zf)(\left. \omega \right|_H) ]^\wedge(\hat{\kappa}). \]
Thus \eqref{eq:ZakVsFib} holds for all $f\in C_c(G)$; extending with continuity gives it for all $f\in L^2(G)$.

\begin{example}
Let us interpret \eqref{eq:ZakVsFib} for the classical Zak transform \eqref{eq:ClasZak}. Identify $\hat{\R}$ with $\R$ and $\Z^*$ with $\Z$ in the usual way: each $\xi \in \R$ defines a character $\hat{\xi} \in \hat{\R}$ by $\hat{\xi}(x) = e^{2\pi i \xi x}$, and $\Z^* = \{ \hat{k}\in \hat{\R} : k \in \Z\}$. In Example \ref{examp:Zaks}(i), the identification of $\T\cong \hat{\Z} \cong \hat{\R}/\Z^*$ with $[0,1)\subset \R$ describes a Borel section $\beta \colon \hat{\R}/\Z^* \to \hat{\R}$ with fundamental domain $\beta(\hat{\R}/\Z^*) = [0,1) \subset \hat{\R}$. Then for $\varphi \in L^2([0,1)\times [0,1))$, the modulation $M$ in \eqref{eq:ZakFibMod} is given by
\[ (M\varphi)(s,t) = \hat{s}(-t) \varphi(s,t) = e^{-2\pi i s t} \varphi(s,t). \]
Thus, \eqref{eq:ZakVsFib} says that for all $f\in L^2(\R)$ and a.e.\ $s\in [0,1)$,
\[ \hat{f}(s+k) = \int_0^1 e^{-2\pi i s t} \cdot (\tilde{Z}f)(s,t)\cdot e^{-2\pi i k t} \, dt = \int_0^1 (\tilde{Z}f)(s,t)\cdot e^{-2\pi i (s+k)t}\, dt  \quad \text{for all }k\in \Z. \]
\end{example}

\smallskip

Here is another relation between the fiberization map and the Zak transform. Fix $f,g \in L^2(G)$. For every $\xi \in H$, \eqref{eq:ZakTrans} and \eqref{eq:ModularI} show that
\begin{equation}\label{eq:ZakBrack}
\langle f, L_\xi g \rangle = \langle Zf, Z(L_\xi g) \rangle = \int_{\hat{H}} \langle (Zf)(\alpha), (Z L_\xi g)(\alpha) \rangle d\mu_{\hat{H}}(\alpha) = \int_{\hat{H}} \langle (Zf)(\alpha), (Zg)(\alpha) \rangle \alpha(\xi)\, d\mu_{\hat{H}}(\alpha)
\end{equation}
\[ = \int_{\hat{H}} \langle (Zf)(\alpha^{-1}), (Zg)(\alpha^{-1}) \rangle \overline{\alpha(\xi)}\, d\mu_{\hat{H}}(\alpha). \]
On the other hand, a similar computation involving \eqref{eq:FibTrans} produces
\begin{equation}\label{eq:FibBrack}
\langle f, L_\xi g \rangle = \int_{\hat{G}/H^*} \langle (\mathcal{T}f)(\omega^{-1} H^*), (\mathcal{T}g)(\omega^{-1} H^*) \rangle \overline{\omega(\xi)}\, d\mu_{\hat{G}/H*}(\omega H^*)
\end{equation}
\[ = \int_{\hat{H}} \langle (\mathcal{T}f)(\omega^{-1} H^*), (\mathcal{T}g)(\omega^{-1} H^*) \rangle \overline{\omega(\xi)}\, d\mu_{\hat{H}}(\left. \omega \right|_H). \]
Since the Fourier transform $L^1(\hat{H}) \to C_0(H)$ is injective,
\begin{equation}
\langle (\mathcal{T}f)(\omega H^*), (\mathcal{T}g)(\omega H^*) \rangle_{L^2(H^*)} = \langle (Zf)(\left. \omega \right|_H), (Zg)(\left. \omega \right|_H) \rangle_{L^2(H\backslash G)} \quad \text{for a.e. }\omega H^* \in \hat{G}/H^*.
\end{equation}

\medskip

%H-TI SPACES=============================================================
\section{The structure of $H$-TI spaces in $L^2(G)$} \label{sec:HTIStr}

Returning to the more general case, where $G$ need not be abelian, we now classify $H$-TI spaces in $L^2(G)$.  Given a family $\A \subset L^2(G)$, we will denote
\[ E^H(\A) = \{ L_\xi \varphi : \xi \in H, \varphi \in \A\} \]
for the left $H$-translates of $\A$, and
\[ S^H(\A) = \overline{\spn}\{ L_\xi \varphi : \xi \in H, \varphi \in \A\} \]
for the $H$-TI space it generates. We will also give conditions under which $E^H(\A)$ forms a continuous frame or a Riesz basis for $S^H(\A)$.

When $J \colon \hat{H} \to \{\text{closed subspaces of }L^2(H\backslash G)\}$ is a range function, we write $P_J(\alpha) \colon L^2(H\backslash G) \to J(\alpha)$ for the orthogonal projection associated to $\alpha \in \hat{H}$. We also denote
\[ V_J = \{ f \in L^2(G) : (Zf)(\alpha) \in J(\alpha) \text{ for a.e.\ } \alpha \in \hat{H}\}. \]
If $G$ is abelian and $\tilde{J} \colon \hat{G}/H^* \to \{ \text{closed subspaces of }L^2(H^*)\}$, we similarly write $\tilde{P}_{\tilde{J}}(\omega H^*) \colon L^2(H^*) \to \tilde{J}(\omega H^*)$ for the orthogonal projection associated to $\omega H^* \in \hat{G}/H^*$, and we define
\[ \tilde{V}_{\tilde{J}} = \{ f \in L^2(G) : (\mathcal{T}f)(\omega H^*) \in \tilde{J}(\omega H^*) \text{ for a.e.\ } \omega H^* \in \hat{G}/H^*\}.\]

The next theorem is an application of \cite[Theorem 2.4]{BR}. Its provenance stretches back to Helson \cite{H} and Srinivasan \cite{S}. Part (ii) generalizes results of de Boor, DeVore, and Ron \cite[Result 1.5]{BDR1}; Bownik \cite[Proposition 1.5]{B}; Cabrelli and Paternostro \cite[Theorem 3.10]{CP}; Kamyabi Gol and Raisi Tousi \cite[Theorem 3.1]{KR}; and Bownik and Ross \cite[Theorem 3.8]{BR}. In contrast with these references, we do not require $G/H$ to be compact. Part (i) opens the door even wider, by allowing $G$ to be nonabelian. As far as the author knows, the results in (i) are new even for $\Z \subset \R$. 

For another description of $H$-TI spaces, in terms of the ``extra'' invariance of an invariant subspace, we refer the reader to \cite{ACHKM,ACP,ACP2,SW}. In the special case where $G$ is abelian and $H$ contains a countable discrete subgroup $K$ such that $G/K$ is compact, these papers describe $H$-TI spaces in terms of the range function classification of $K$-TI spaces given in \cite{B,CP,BDR1,KR}. In particular, their descriptions of $H$-TI spaces use the fiberization map for $K\subset G$. We do not require $H$ to contain such a subgroup here, and our classifications are in terms of the Zak transform and fiberization map for $H$ itself.

\begin{theorem}\label{thm:HTIClass}
(i) $H$-TI spaces in $L^2(G)$ are indexed by measurable range functions 
\[ J\colon \hat{H} \to \{ \text{closed subspaces of }L^2(H\backslash G)\}, \]
provided we identify range functions that agree a.e. A bijection maps $J \mapsto V_J$. When $\A \subset L^2(G)$ is a family with a countable dense subset $\A_0 \subset \A$, $S^H(\A) = S^H(\A_0)$, and the associated range function is given by
\begin{equation}\label{eq:ZakJ}
J(\alpha) = \overline{\spn}\{ (Z f)(\alpha) : f \in \A_0\}.
\end{equation}

\medskip

\noindent (ii) In addition to the standing assumptions, suppose that $G$ is abelian. Then $H$-TI spaces in $L^2(G)$ can also be indexed by measurable range functions
\[ \tilde{J} \colon \hat{G}/H^* \to \{\text{closed subspaces of }L^2(H^*)\}, \]
provided we identify range functions that agree a.e. A bijection maps $\tilde{J} \mapsto \tilde{V}_{\tilde{J}}$. For a family $\A \subset L^2(G)$ with countable dense subset $\A_0 \subset \A$, the range function associated with $S^H(\A) = S^H(\A_0)$ is
\begin{equation}\label{eq:FibJ}
\tilde{J}(\omega H^*) = \overline{\spn}\{ (\mathcal{T} f)(\omega H^*) : f \in \A_0\}.
\end{equation}
\end{theorem}

We will need the following lemma, which essentially restates \cite[Lemma 3.5]{BR}.

\begin{lemma}\label{lem:CharParDet}
Let $\mathcal{G}$ be an LCA group with Haar measure $\mu_{\mathcal{G}}$, and let $\hat{\mathcal{G}}$ be its dual group with dual Haar measure $\mu_{\hat{\mathcal{G}}}$. Then $\hat{\mathcal{G}}$ forms a Parseval determining set  for $L^1(\mathcal{G})$ with respect to $\mu_{\hat{\mathcal{G}}}$. In other words,
\begin{equation}\label{eq:Planch}
\int_{\hat{\mathcal{G}}} \left| \int_{\mathcal{G}} f(x) \overline{\alpha(x)}\, d\mu_{\mathcal{G}}(x) \right|^2 d\mu_{\hat{\mathcal{G}}}(\alpha) = \int_{\mathcal{G}} |f(x)|^2\, d\mu_{\mathcal{G}}(x)
\end{equation}
for each $f\in L^1(\mathcal{G})$; both sides may be infinite.
\end{lemma}

\begin{proof}
For $f\in L^1(\mathcal{G})$, the left hand side of \eqref{eq:Planch} is precisely $\Norm{\hat{f}}_2^2$. If $f \in L^1(\mathcal{G}) \cap L^2(\mathcal{G})$, \eqref{eq:Planch} is just Plancherel's Theorem. On the other hand, if $\Norm{f}_2 = \infty$, then $\Norm{\hat{f}}_2 = \infty$ by 31.44(a) of \cite{HR2}. 
\end{proof}

\begin{rem}\label{rem:CharParDet}
Pontryagin Duality allows us to switch $\G$ and $\hat{\G}$ in the lemma above. Given $x \in \G$, write $X_x \in \hat{\hat{\G}}$ for the corresponding character $X_x(\alpha) = \alpha(x)$. Then $\D = (X_x)_{x \in \G}$ is a Parseval determining set for $L^1(\hat{\G})$ with respect to $\mu_\G$. 

When $G$ is abelian, we can identify $(\hat{G}/H^*, \mu_{\hat{G}/H^*})$ with $(\hat{H}, \mu_{\hat{H}})$ by mapping $\omega H^* \mapsto \left. \omega \right|_H$. Each $\xi \in H$ then defines a character $\tilde{X}_\xi$ on $\hat{G}/H^*$ by the formula 
\[ \tilde{X}_\xi(\omega H^*) = X_\xi(\left. \omega \right|_H) = \omega(\xi), \]
and the previous paragraph shows that $\tilde{\D} = ( \tilde{X}_\xi )_{\xi \in H}$ is a Parseval determining set for $L^1(\hat{G}/H^*)$ with respect to $\mu_H$.
\end{rem}

\smallskip

\noindent \emph{Proof of Theorem \ref{thm:HTIClass}.}
Let $\D$ be as in Remark \ref{rem:CharParDet}, with $\G = H$. By Theorem \ref{thm:Zak}, a subspace $M\subset L^2(G)$ is $H$-TI if and only if $ZM$ is a $\D$-MI subspace of $L^2(\hat{H}; L^2(H\backslash G))$. Thus (i) is an application of Proposition \ref{prop:AbstRan}. Likewise, (ii) follows immediately from Proposition \ref{prop:Fib}, Proposition \ref{prop:AbstRan}, and the remark above.
\hfill \qed

\medskip

As in the familiar case of integer shifts in $L^2(\R^n)$, our classification of $H$-TI spaces ties with a set of conditions under which the $H$-translates of a family $\A \subset L^2(G)$ form a continuous frame. Namely, it reduces the problem of $E^H(\A)$ forming a continuous frame for $S^H(\A)$ to an analysis of the fibers $J(\alpha) = \overline{\spn}\{(Zf)(\alpha) : f\in \A\}$. If $G$ is abelian we can replace the Zak transform with fiberization, and if $H$ is discrete we can replace ``continuous frame'' with ``Riesz basis''.

The next two theorems are applications of Theorems \ref{thm:Riesz} and \ref{thm:frame}. They generalize results of  Bownik \cite[Theorem 2.3]{B}; Kamyabi Gol and Raisi Tousi \cite[Theorems 4.1 and 4.2]{KR}; Cabrelli and Paternostro \cite[Theorems 4.1 and 4.3]{CP}; and Bownik and Ross \cite[Theorem 5.1]{BR}. In contrast with these results, we do not require $G/H$ to be compact. When we use the Zak transform, we do not even need $G$ to be abelian.

\begin{theorem}\label{thm:HTIFrame}
Let $(\mathcal{M},\mu_\mathcal{M})$ be a complete, $\sigma$-finite measure space, and let $\A = (f_t)_{t\in \mathcal{M}} \subset L^2(G)$ be a jointly measurable family of functions. Fix a countable dense subset $\A_0 \subset \A$, and let $J$ be as in \eqref{eq:ZakJ}. Given constants $0 < A \leq B < \infty$, the following are equivalent:
\begin{enumerate}[(i)]

\item $E^H(\A)$ forms a continuous frame for $S^H(\A)$ over $\mathcal{M} \times H$, with bounds $A,B$. In other words, for every $g\in S^H(\A)$,
\[ A \int_G |g(x)|^2\, d\mu_G(x) \leq \int_\mathcal{M} \int_H \left| \int_G g(x) \overline{ L_\xi f_t(x) }\, d\mu_G(x) \right|^2 d\mu_H(\xi)\, d\mu_\mathcal{M}(t) \leq B \int_G |g(x)|^2\, d\mu_G(x). \]

\item For a.e.\ $\alpha \in \hat{H}$ and every $h \in J(\alpha) \subset L^2(H\backslash G)$,
\[ A \Norm{h}^2 \leq \int_\mathcal{M} \left| \langle h, (Zf_t)(\alpha) \rangle \right|^2\, d\mu_\mathcal{M}(t) \leq B \Norm{h}^2. \]
\end{enumerate}
If $G$ is abelian and $\tilde{J}$ is as in \eqref{eq:FibJ}, the conditions above are equivalent to:

\begin{enumerate}[(i)]
\setcounter{enumi}{2}

\item For a.e.\ $\omega H^* \in \hat{G} / H^*$ and every $h \in \tilde{J}(\omega H^*) \subset L^2(H^*)$,
\[ A \Norm{h}^2 \leq \int_\mathcal{M} \left| \langle h, (\mathcal{T}f_t)(\omega H^*) \rangle \right|^2\, d\mu_\mathcal{M}(t) \leq B \Norm{h}^2. \]
\end{enumerate}
\end{theorem}

As in the remarks following Theorem \ref{thm:frame}, condition (ii) says that for a.e.\ $\alpha \in \hat{H}$, the family $\{ [P_J(\alpha)] (Z f_t)(\alpha) : t\in \mathcal{M}\}$ forms a continuous frame for $J(\alpha)$ with bounds $A,B$. A similar consideration applies to (iii). 

When $\A$ is countable this theorem reduces a continuous problem in $L^2(G)$ to a discrete problem in $L^2(H\backslash G)$ or $L^2(H^*)$. For instance, when $\A$ consists of a single function $f\in L^2(G)$, condition (ii) is equivalent to
\begin{enumerate}[(i')]
\setcounter{enumi}{1}
\item For a.e.\ $\alpha \in \hat{H}$, either $(Zf)(\alpha) = 0$ or $A \leq \Norm{(Zf)(\alpha)}^2 \leq B$.
\end{enumerate}

\medskip

\noindent \emph{Proof of Theorem \ref{thm:HTIFrame}.}
We claim that $Z \A = (Z f_t)_{t\in \mathcal{M}} \subset L^2(\hat{H}; L^2(H\backslash G))$ is jointly measurable. To prove this, we consider the image of $\A$ under each of the isomorphisms $U_k$ used to construct $Z$ in the proof of Theorem \ref{thm:Zak}. The first isomorphism $U_1 \colon L^2(G) \to L^2(H\times H \backslash G)$ is gotten from a measure space isomorphism, so it must preserve the notion of joint measurability. Corollary \ref{cor:JointMeas} shows joint measurability is preserved by $U_2 \colon L^2(H\times H\backslash G) \to L^2(H\backslash G; L^2(H))$. Since the Fourier transform $L^2(H) \to L^2(\hat{H})$ leaves inner products unchanged, joint measurability is preserved by $U_3 : L^2(H\backslash G; L^2(H)) \to L^2(H\backslash G; L^2(\hat{H}))$. Another application of Corollary \ref{cor:JointMeas} gives joint measurability after applying $U_4 \colon L^2(H\backslash G; L^2(\hat{H})) \to L^2(\hat{H}; L^2(H\backslash G))$. This proves the claim.

Let $\D = (X_\xi)_{\xi \in H}$ be the Parseval determining set from Remark \ref{rem:CharParDet}.  Since the unitary $Z \colon L^2(G) \to L^2(\hat{H}; L^2(H\backslash G))$ intertwines left translation by $\xi \in H$ with multiplication by $X_\xi \in \D$, condition (i) above is equivalent to:
\begin{enumerate}[(i')]
\item $E_\D(Z\A)$ forms a continuous frame for $S_\D(Z\A)$ with bounds $A,B$.
\end{enumerate}
Moreover, the range function associated with the $\D$-MI space $S_\D(Z\A)$ is precisely $J$, by Proposition \ref{prop:AbstRan}. Hence, the equivalence of (i) and (ii) follows from the corresponding equivalence in Theorem \ref{thm:frame}.

When $G$ is abelian, the fiberization map $\mathcal{T}$ is made by composing the Fourier transform $L^2(G) \to L^2(\hat{G})$ with the isomorphisms $L^2(\hat{G}) \to L^2(H^* \times \hat{G}/H^*)$ and $L^2(H^* \times \hat{G}/H^*) \to L^2(\hat{G}/H^*; L^2(H^*))$. The first isomorphism preserves joint measurability as an easy consequence of Proposition \ref{prop:JointMeas} and Plancherel's Theorem, the second preserves it because it is based on a measure space isomorphism, and the third preserves it by Corollary \ref{cor:JointMeas}. Consequently, $\mathcal{T} \A = (\mathcal{T} f_t)_{t\in \mathcal{M}} \subset L^2(\hat{G}/H^*; L^2(H^*))$ is jointly measurable. An argument similar to the one in the paragraph above now proves the equivalence of (i) and (iii): replace $Z$ with $\mathcal{T}$, and $\D$ with $\tilde{\D}$ from Remark \ref{rem:CharParDet}.
\hfill \qed

\begin{theorem}\label{thm:HTIRiesz}
In addition to the standing assumptions, suppose that $H$ is discrete and $\mu_H$ is counting measure. Let $\A \subset L^2(G)$ be a countable family, and let 
\[ J(\alpha) = \overline{\spn}\{ (Zf)(\alpha) : f \in \A\} \]
for a.e.\ $\alpha \in \hat{H}$. For constants $0 < A \leq B < \infty$, the following are equivalent:
\begin{enumerate}[(i)]
\item $E^H(\A)$ is a Riesz basis for $S^H(\A)$ with bounds $A,B$.
\item For a.e.\ $\alpha \in \hat{H}$, $\{ (Zf)(\alpha) : f\in \A\}$ is a Riesz basis for $J(\alpha)$ with bounds $A,B$.
\end{enumerate}
If $G$ is abelian and
\[ \tilde{J}(\omega H^*) = \overline{\spn}\{ (\mathcal{T} f)(\omega H^*) : f\in \A\} \]
for a.e.\ $\omega H^* \in \hat{G}/ H^*$, the conditions above are equivalent to:
\begin{enumerate}[(i)]
\setcounter{enumi}{2}
\item For a.e.\ $\omega H^* \in \hat{G}/H^*$, $\{ (\mathcal{T}f)(\omega H^*) : f\in \A\}$ is a Riesz basis for $\tilde{J}(\omega H^*)$ with bounds $A,B$.
\end{enumerate}
\end{theorem}

\begin{proof}
Recall that discrete abelian groups are dual to compact abelian groups, with counting measures dual to probability measures. Hence $\mu_{\hat{H}}(\hat{H}) = 1$. The theorem now follows from Theorem \ref{thm:Riesz} in the same way that Theorem \ref{thm:HTIFrame} followed from Theorem \ref{thm:frame}.
\end{proof}

Strictly speaking, the previous theorem holds even if $H$ is not discrete. However, when $\hat{H}$ is not compact, condition (ii) can never occur. See Remark \ref{rem:Riesz}.

\begin{rem}
When $\A$ consists of a single function $f\in L^2(G)$, the conditions in the previous theorems simplify even further. Let $\Omega_f = \{ \alpha \in \hat{H} : (Zf)(\alpha) \neq 0\}$. Then condition (ii) of Theorem \ref{thm:HTIFrame} is equivalent to
\begin{enumerate}[(ii')]
\item For a.e.\ $\alpha \in \Omega_f$, $A \leq \Norm{(Zf)(\alpha)}^2 \leq B$.
\end{enumerate}
When $H$ is discrete, we can likewise replace condition (ii) of Theorem \ref{thm:HTIRiesz} with
\begin{enumerate}[(ii')]
\item For a.e.\ $\alpha \in \hat{H}$, $A \leq \Norm{(Zf)(\alpha)}^2 \leq B$.
\end{enumerate}
Similar considerations apply for fiberization in the abelian setting.
\end{rem}

We end this section with a pair of results on Gabor systems with critical sampling. We will assume that $G$ is abelian. A closed subspace $M\subset L^2(G)$ is called \emph{$(H,H^*)$-translation/modulation-invariant}, or $(H,H^*)$-TMI, if $L_\xi M_\kappa f \in M$ whenever $f\in M$, $\xi \in H$, and $\kappa \in H^*$. TMI spaces have usually been called ``shift/modulation invariant'', or SMI, in the discrete case. Following the examples of Bownik and Ross \cite{BR} and Jakobsen and Lemvig \cite{JL}, we adopt the term TMI to emphasize that the subgroup involved need not be discrete.

Every family $\A \subset L^2(G)$ generates a \emph{Gabor system} $\{ L_\xi M_\kappa f : \xi \in H, \kappa \in H^*, f \in \A\}$. The closed linear span of this system is the smallest $(H,H^*)$-TMI space containing $\A$. The Zak transform has a long history of use for Gabor systems. We continue the tradition here. Our first result classifies $(H,H^*)$-TMI spaces in terms of the Zak transform. Our second result tells when Gabor systems are continuous frames.

The theorem below should be compared with Bownik \cite[Theorem 5.1]{B2} and Cabrelli and Paternostro \cite[Theorem 5.1]{CP2}. Given a Borel subset $E \subset \hat{H} \times H\backslash G$, we denote
\[ M_E = \{ f \in L^2(G) : (\tilde{Z}f)(\alpha, Hx) = 0 \text{ for a.e.\ }(\alpha,Hx) \notin E\}. \]
Two Borel subsets of $\hat{H} \times H\backslash G$ are called \emph{equivalent} if their symmetric difference has measure zero.

\begin{theorem}\label{thm:TMIClass}
The $(H,H^*)$-TMI spaces in $L^2(G)$ are indexed by equivalence classes of Borel subsets of $\hat{H} \times H\backslash G$. A bijection maps $E \mapsto M_E$. For $\A \subset L^2(G)$, any countable dense subset $\A_0\subset \A$ generates the same $(H,H^*)$-TMI space as does $\A$, and the corresponding subset of $\hat{H} \times H\backslash G$ is
\begin{equation}\label{eq:TMISet}
E = \{ (\alpha, Hx) \in \hat{H} \times H\backslash G : (\tilde{Z} f)(\alpha, Hx) \neq 0 \text{ for some }f \in \A_0\}.
\end{equation}
\end{theorem}

\begin{proof}
Since $H^* \cong (H\backslash G)^\wedge$, Pontryagin duality shows that $(\hat{H} \times H\backslash G)^\wedge \cong H\times H^*$. For $(\xi, \kappa) \in H \times H^*$, the corresponding character $X_{(\xi,\kappa)}\in (\hat{H} \times H\backslash G)^\wedge$ is given by $X_{(\xi, \kappa)}(\alpha, Hx) = \alpha(\xi) \kappa(x)$. By Lemma \ref{lem:CharParDet}, the family $\D=(X_{(\xi,\kappa)})_{\xi \in H, \kappa \in H^*}$ is a Parseval determining set for $L^1(\hat{H} \times H\backslash G)$. Moreover, \eqref{eq:ZakTransMod} shows that a subspace $M\subset L^2(G)$ is $(H,H^*)$-TMI if and only if $\tilde{Z}M \subset L^2(\hat{H} \times H\backslash G)$ is $\D$-MI. The proof follows from Proposition \ref{prop:AbstRan} once we observe that a range function $J\colon \hat{H} \times H\backslash G \to \{ \text{closed subsets of }\C\}$ identifies uniquely with the set
\[ E = \{ (\alpha, Hx) \in \hat{H} \times H\backslash G : J(\alpha,Hx) = \C\}. \]
Moreover, $J$ is a measurable range function if and only if $E$ is a Borel set.
\end{proof}

The next theorem generalizes a result of Arefijamaal \cite[Theorem 2.6]{A}. Also see Corollary 6.4.4 of Gr\"{o}chenig \cite{Gr}, and the discussion that follows it.

\begin{theorem}\label{thm:TMIFrame}
Let $(\mathcal{M},\mu_\mathcal{M})$ be a complete, $\sigma$-finite measure space, and let $\A = (f_t)_{t\in \mathcal{M}}\subset L^2(G)$ be a jointly measurable family of functions. Fix a countable dense subset $\A_0 \subset \A$, and let $E\subset \hat{H} \times H\backslash G$ be as in \eqref{eq:TMISet}. For constants $0 < A \leq B < \infty$, the following are equivalent:
\begin{enumerate}[(i)]
\item The Gabor system generated by $\A$ is a continuous frame for its closed linear span, with bounds $A,B$.
\item For a.e.\ $(\alpha, Hx) \in E$, 
\[ A \leq \int_\mathcal{M} |(\tilde{Z} f_t)(\alpha, Hx)|^2\, d\mu_\mathcal{M}(t) \leq B. \]
\end{enumerate}
\end{theorem}

\begin{proof}
As in the proof of Theorem \ref{thm:HTIFrame}, the family $\tilde{Z} \A = (\tilde{Z} f_t)_{t \in \mathcal{M}} \subset L^2(\hat{H} \times H\backslash G)$ is jointly measurable. The theorem now follows from Theorem \ref{thm:frame} in the same way that Theorem \ref{thm:TMIClass} followed from Proposition \ref{prop:AbstRan}.
\end{proof}

\medskip

%DUAL INTEGRABLE REPS===================================================
\section{Dual integrable representations of LCA groups}\label{sec:DualInt}

We now turn our attention to a more general problem. Given a representation of a locally compact group on a Hilbert space $\H$, we would like to know when the orbit of a family of vectors $\A \subset \H$ makes a continuous frame in $\H$. We give an answer for a large class of representations of LCA groups.

Throughout this section, $\G$ will denote a fixed, second countable LCA group. Its Haar measure is $\mu_\G$, its dual group is $\hat{\G}$, and the dual Haar measure on $\hat{\G}$ is $\mu_{\hat{\G}}$. For $x\in \G$, the corresponding character of $\hat{\G}$ is $X_x$; that is, $X_x(\alpha) = \alpha(x)$. We set $\D = (X_x)_{x\in G}$. As explained in Remark \ref{rem:CharParDet}, $\D$ is a Parseval determining set for $L^1(\hat{\G})$.

A (unitary) \emph{representation} of $\G$ on a Hilbert space $\H$ is a strongly continuous group homomorphism $\pi \colon \G \to U(\H)$ into the unitary group of $\H$. We call $\pi$ \emph{dual integrable} if there is a function
\[ [ \cdot , \cdot ] \colon \H \times \H \to L^1(\hat{\G}), \]
called a \emph{bracket} for $\pi$, such that
\[ \langle \varphi, \pi(x) \psi \rangle = \int_{\hat{\G}} [\varphi,\psi](\alpha)\cdot \overline{\alpha(x)}\, d\mu_{\hat{\G}}(\alpha) \quad \text{for all }\varphi, \psi \in \H \text{ and }x \in \G. \]
In other words, a representation is dual integrable when all of its matrix elements lie in the image of the Fourier transform $L^1(\hat{\G}) \to C_0(\G)$. The bracket gives the inverse Fourier transform of a matrix element. Consequently, the bracket is unique when it exists.

Dual integrable representations were introduced in the abstract setting by Hern{\'a}ndez, \v{S}iki{\'c}, Weiss, and Wilson in \cite{HSWW}. Concrete versions of the bracket have been around much longer. Early uses appear in Jia and Michelli \cite{JM} and de Boor, DeVore, and Ron \cite{BDR1,BDR2}. An analog of dual integrable representations for possibly nonabelian countable discrete groups was recently developed by Barbieri, Hern{\'a}ndez, and Parcet in \cite{BHP}. Another version for square integrable functions over the Heisenberg group appears in Barbieri, Hern{\'a}ndez, and Mayeli \cite{BHM}. 

In this section, we fix a dual integrable representation $\pi$ acting on a \emph{separable} Hilbert space $\H$. Given $\varphi \in \H$, we denote
\[ \langle \varphi \rangle = \overline{\spn}\{ \pi(x) \varphi : x \in \G \}. \]
We begin by recalling some basic properties of the bracket from \cite{HSWW}.

\begin{prop}\label{prop:BrackProp}
The bracket is a sesquilinear Hermitian map $[\cdot,\cdot]\colon \H \times \H \to L^1(\hat{G})$. Moreover, for $\varphi, \psi \in \H$ and $x\in G$, the following hold:
\begin{enumerate}[(i)]
\item $[\varphi,\varphi] \geq 0$ a.e.
\item $| [\varphi, \psi] | \leq [\varphi,\varphi]^{1/2} [\psi, \psi]^{1/2}$ a.e.
\item $\varphi \perp \langle \psi \rangle$ if and only if $[\varphi,\psi]=0$ a.e.
\item $[\pi(x) \varphi, \psi] = X_x\cdot [\varphi,\psi] = [\varphi, \pi(x^{-1}) \psi]$
\end{enumerate}
\end{prop}

Our strategy for understanding the translation action of an abelian subgroup in Section \ref{sec:HTIStr} was to apply an isometry that intertwined that action with modulation. We employ the same method here. Our isometry will be based on the following notion. The terminology is our own invention.

\begin{defn}
Let $\pi$ be a representation of $\G$ on a Hilbert space $\H$. A family of vectors $(\theta_i)_{i\in I} \subset \H$ is called \emph{orthogonal generators} for $\pi$ if $\H = \bigoplus_{i\in I} \langle \theta_i \rangle$.
\end{defn}

Every representation admits a family of orthogonal generators, as a well-known consequence of Zorn's Lemma. Normally, the choice of generators is far from unique. For instance, any family of functions $\{f_i\}_{i\in I} \subset L^2(\R)$ for which $\{ \supp \hat{f_i} \}_{I\in I}$ forms a partition of $\R$, is an orthogonal generating family of the regular representation of $\R$. As this example demonstrates, orthogonal generators abound, and the cardinality of the indexing set $I$ can change dramatically from family to family. 

In an abstract sense, the lack of a canonical family of orthogonal generators might seem annoying, but in a practical sense, it is an advantage. In what follows, we analyze a dual integrable representation in terms of its bracket, a family $(\theta_i)_{i\in I}$ of orthogonal generators, and $l^2(I)$. The abundance of orthogonal generating families only makes this analysis more flexible.

For the remainder of the paper, we fix a family $(\theta_i)_{i\in I} \subset \H$ of orthogonal generators for $\pi$. For $i\in I$, we denote
\[ \Omega_i = \{ \alpha \in \hat{\G} : [\theta_i,\theta_i] \neq 0\}. \]
We also write $\delta_i \in l^2(I)$ for the standard basis element corresponding to $i\in I$.

The next proposition is Corollary (3.2) of \cite{HSWW}. The corollary after it was partially explained in the proof of \cite[Corollary (3.4)]{HSWW}.

\begin{prop}\label{prop:CycIsom}
Let $\psi \in \H$, and denote
\begin{equation}\label{eq:OmegaPsi}
\Omega_\psi = \{ \alpha \in \hat{\G} : [\psi,\psi](\alpha) \neq 0\},
\end{equation}
which is well defined up to a set of measure zero. The function $T_\psi \colon \langle \psi \rangle \to  L^2(\hat{\G})$ given by
\begin{equation}\label{eq:BrackIsomDef}
T_\psi(\varphi) = \mathbf{1}_{\Omega_\psi} \frac{[\varphi,\psi]}{[\psi,\psi]^{1/2}} \quad \text{for }\varphi \in \langle \psi \rangle \subset \H
\end{equation}
maps $\langle \psi \rangle$ unitarily onto $L^2(\Omega_\psi, \mu_{\hat{\G}})$.
\end{prop}

\begin{cor} \label{cor:BrackIsom}
The function $T \colon \H \to L^2(\hat{\G};l^2(I))$ given by
\[ T(\varphi)(\alpha) = \left( \mathbf{1}_{\Omega_i}(\alpha) \cdot \frac{[\varphi, \theta_i](\alpha)}{ ( [\theta_i, \theta_i](\alpha) )^{1/2}} \right)_{i\in I} \quad \text{for }\varphi \in \H \text{ and } \alpha \in \hat{\G} \]
is a linear isometry satisfying
\begin{equation}\label{eq:BrackIsomMod}
T(\pi(x) \varphi) = X_x \cdot T(\varphi) \quad \text{for all }\varphi \in \H \text{ and }x \in \G.
\end{equation}

In particular, $T(\H)$ is a $\D$-MI space in $L^2(\hat{\G}; l^2(I))$. The range function $J_0 \colon \hat{\G} \to \{ \text{closed subspaces of }l^2(I)\}$ given by
\[ J_0(\alpha) = \overline{\spn}\{ \mathbf{1}_{\Omega_i}(\alpha) \cdot \delta_i : i \in I \} \]
corresponds to $T(\H)$, in the sense of Proposition \ref{prop:AbstRan}(ii).
\end{cor}

\begin{proof}
For each $i \in I$, let $P_i \colon \H \to \langle \theta_i \rangle$ be orthogonal projection, and let $T_i=T_{\theta_i}\colon \langle \theta_i \rangle \to L^2(\hat{\G})$ be the map from Proposition \ref{prop:CycIsom}. Given $\varphi \in \H$ and $i \in I$, Proposition \ref{prop:BrackProp} implies that
\[ [ \varphi, \theta_i](\alpha) = [P_i \varphi, \theta_i ](\alpha) + [(1-P_i) \varphi, \theta_i ](\alpha) = [P_i \varphi, \theta_i ](\alpha). \]
Consequently,
\[ T_i P_i \varphi = \mathbf{1}_{\Omega_i} \frac{[\varphi, \theta_i]}{([\theta_i, \theta_i])^{1/2}}; \]
in other words,
\[ T(\varphi)(\alpha) = ( (T_i P_i \varphi)(\alpha) )_{i\in I}. \]
By Proposition \ref{prop:CycIsom}, $T$ maps the spaces $\langle \theta_i \rangle\subset \H$ isometrically into orthogonal subspaces of $L^2(\hat{G}; l^2(I))$. Since $\H = \bigoplus_{i\in I} \langle \theta_i \rangle$, $T$ is a linear isometry. Proposition \ref{prop:BrackProp}(iv) gives \eqref{eq:BrackIsomMod}.

For $i,j \in I$, Proposition \ref{prop:BrackProp} quickly implies that $[\theta_i, \theta_j] = \delta_{i,j} \cdot [\theta_i, \theta_i]$, where $\delta_{i,j}$ is the Kronecker-delta. Thus,
\begin{equation}\label{eq:BrackIsomGen}
T(\theta_i)(\alpha) = ( [\theta_i, \theta_i](\alpha) )^{1/2} \cdot \delta_i.
\end{equation}
Since $\H = \overline{\spn}\{ \pi(x) \theta_i : x \in \G, i \in I\}$, we have, in the language of Section \ref{sec:frames},
\[ T(\H) = \overline{\spn}\{ T(\pi(x) \theta_i) : x\in \G, i \in I\} = \overline{\spn}\{ X_x \cdot T(\theta_i) : x\in \G, i \in I\} = E_\D( \{T(\theta_i)\}_{i\in I}). \]
By Proposition \ref{prop:AbstRan}(iii), the range function associated with $T(\H)$ is
\[ J_0(\alpha) = \overline{\spn}\{ T(\theta_i)(\alpha) : i \in I \} =\overline{\spn}\{ \mathbf{1}_{\Omega_i}(\alpha)\cdot \delta_i : i\in I\}. \]
\end{proof}

A closed subspace $M \subset \H$ is called \emph{$\pi$-invariant} if $\pi(x) \varphi \in M$ whenever $\varphi \in M$ and $x\in \G$. The restriction of each $\pi(x)$ to $M$ gives the \emph{subrepresentation} of $\pi$ on $M$. The subrepresentation is also dual integrable, with the same bracket. The next theorem classifies $\pi$-invariant subspaces of $\H$ in terms of range functions.

Given a range function $J\colon \hat{\G} \to \{\text{closed subspaces of }l^2(I)\}$, we denote $P_J(\alpha) \colon l^2(I) \to J(\alpha)$ for the orthogonal projection associated to $\alpha \in \hat{\G}$. We also write
\[ V_J = \{ \varphi \in \H : (T\varphi)(\alpha) \in J(\alpha) \text{ for a.e.\ }\alpha \in \hat{\G}\}. \]
We call two range functions \emph{equivalent} when they agree a.e.\ on $\hat{\G}$.

Given a family $\A \subset \H$, we write
\[ E(\A) = \{\pi(x) \varphi : x \in \G, \varphi \in \A\} \]
for its orbit under $\pi$, and
\[ S(\A) = \overline{\spn}\{ \pi(x) \varphi : x \in \G, \varphi \in \A\} \]
for the $\pi$-invariant space it generates.

\begin{theorem}\label{thm:DualClas}
Let $J_0$ be as in Corollary \ref{cor:BrackIsom}. The $\pi$-invariant subspaces of $\H$ are indexed by equivalence classes of measurable range functions $J \colon \hat{\G} \to \{\text{closed subspaces of }l^2(I)\}$ satisfying 
\begin{equation} \label{eq:RanCon}
J(\alpha) \subset J_0(\alpha) \quad \text{for a.e.\ }\alpha \in \hat{\G}.
\end{equation}
A bijection maps $J \mapsto V_J$. 

If $\A \subset \H$ has a countable dense subset $\A_0 \subset \A$, then the range function $J\colon \hat{\G} \to \{\text{closed subspaces of }l^2(I)\}$ given by
\begin{equation}\label{eq:DualClas2}
J(\alpha) = \overline{\spn}\{ (T\varphi)(\alpha) : \varphi \in \A_0\}.
\end{equation}
satisfies
\[ V_J = S(\A_0) = S(\A). \]
\end{theorem}

\begin{proof}
By Corollary \ref{cor:BrackIsom}, $E \mapsto T(E)$ is a bijection between closed $\pi$-invariant subspaces of $\H$ and $\D$-MI spaces contained in $T(\H)$. Moreover, $E=V_J$ if and only if $T(E) = M_J$, in the language of Proposition \ref{prop:AbstRan}. Obviously $M_J \subset T(\H) = M_{J_0}$ if and only if $J$ satisfies \eqref{eq:RanCon}, so the theorem is a consequence of Proposition \ref{prop:AbstRan} and Remark \ref{rem:CharParDet}.
\end{proof}

Representations of LCA groups are uniquely determined by associated projection-valued measures on the dual group. For background, we refer the reader to Folland \cite[Sections 1.4 and 4.4]{F}. Hern{\'a}ndez et al.\ \cite[Corollary (2.5)]{HSWW} have given the projection-valued measure associated with a dual integrable representation, in terms of the bracket. The next proposition gives the projection-valued measure associated with an invariant subspace of a dual integrable representation, in terms of $T$.

\begin{prop}
Let $J\colon \hat{\G} \to \{ \text{closed subspaces of }l^2(I)\}$ be a measurable range function satisfying \eqref{eq:RanCon}. For each $E\subset \hat{\G}$, define a projection $P(E)$ on $V_J$ by the formula
\[ T(P(E) \varphi) = \mathbf{1}_E\cdot T(\varphi). \]
Then $P$ is a regular $V_J$-projection-valued measure on $\hat{\G}$, and the subrepresentation of $\pi$ on $V_J$ is given by
\[ \pi(x) = \int_{\hat{\G}} \alpha(x)\, dP(\alpha). \]
\end{prop}

\begin{proof}
For each $\varphi, \psi \in \H$, define a complex-valued measure $P_{\varphi,\psi}$ on $\hat{\G}$ with the formula
\[ P_{\varphi, \psi}(E) = \langle P(E) \varphi, \psi \rangle = \langle \mathbf{1}_E\cdot T \varphi, T \psi \rangle = \int_{\hat{G}} \mathbf{1}_E(\alpha)\cdot \langle (T \varphi)(\alpha), (T \psi)(\alpha) \rangle\, d\mu_{\hat{\G}}(\alpha) = \int_{\hat{G}} \mathbf{1}_E(\alpha) \cdot [ \varphi, \psi ](\alpha)\, d\mu_{\hat{\G}}(\alpha). \]
In other words, $dP_{\varphi, \psi} = [\varphi, \psi]\, d\mu_{\hat{\G}}$. By Corollary (2.5) of \cite{HSWW},
\[ \langle \pi(x) \varphi, \psi \rangle = \int_{\hat{\G}} \alpha(x)\cdot [\varphi, \psi](\alpha)\, d\mu_{\hat{\G}}(\alpha) = \int_{\hat{\G}} \alpha(x)\, dP_{\varphi,\psi}(\alpha). \]
This completes the proof.
\end{proof}

We now give the main results of this section, reducing frame and Riesz basis conditions on the orbit of a family $\A \subset \H$ to pointwise conditions on the fibers $J(\alpha)$ from \eqref{eq:DualClas2}. In the special case of a discrete LCA group with a cyclic dual integrable representation, the next two theorems were given by Hern{\'a}ndez et al.\ \cite[Proposition (5.3) and Theorem (5.7)]{HSWW}.

\begin{theorem}\label{thm:BrackFrame}
Let $(\mathcal{M},\mu_\mathcal{M})$ be a complete, $\sigma$-finite measure space, and let $\A = (\varphi_t)_{t\in \mathcal{M}} \subset \H$ be a family of vectors such that, for each $i\in I$, the function
\[ (t, \alpha) \mapsto [\varphi_t,\theta_i](\alpha) \]
is measurable on $\mathcal{M} \times \hat{\G}$. Let $\A_0 \subset \A$ be a countable dense subset, and let $J$ be as in \eqref{eq:DualClas2}. For constants $0 < A \leq B < \infty$, the following are equivalent:
\begin{enumerate}[(i)]
\item $E(\A)$ is a continuous frame for $S(\A)$ with bounds $A,B$. That is,
\[ A \Norm{\psi}^2 \leq \int_\mathcal{M} \int_\G | \langle \psi, \pi(x) \varphi_t \rangle |^2\, d\mu_{\G}(x)\, d\mu_\mathcal{M}(t) \leq B \Norm{\psi}^2 \]
for all $\psi \in S(\A)$.
\item For a.e.\ $\alpha \in \hat{\G}$, $\{P_J(\alpha)[T \varphi_t(\alpha)] : t \in \mathcal{M}\}$ is a continuous frame for $J(\alpha)$ with bounds $A,B$. In other words,
\[ A \Norm{v}_{l^2(I)}^2 \leq \int_\mathcal{M} | \langle v, T \varphi_t(\alpha) \rangle_{l^2(I)} |^2\, d\mu_\mathcal{M}(t) \leq B \Norm{v}_{l^2(I)}^2 \]
for a.e.\ $\alpha \in \hat{\G}$ and all $v \in J(\alpha)$.
\end{enumerate}
\end{theorem}

\begin{proof}
By Corollary \ref{cor:BrackIsom}, the linear isometry $T \colon \H \to L^2(\hat{\G}; l^2(I))$ maps $S(\A)$ unitarily onto $S_\D(T \A)$, sending $E(\A)$ to $E_\D(T \A)$. For each $i\in I$, the function
\[ (t, \alpha) \mapsto \mathbf{1}_{\Omega_i}(\alpha) \cdot \frac{[\varphi_t,\theta_i](\alpha)}{([\theta_i,\theta_i](\alpha))^{1/2}} \]
is measurable on $\mathcal{M} \times \hat{\G}$. Therefore
\[ (t, \alpha, i) \mapsto ([T \varphi_t](\alpha))_i \]
is measurable on $\mathcal{M} \times \hat{\G} \times I$. By Corollary \ref{cor:JointMeas}, the family $T \A = (T \varphi_t)_{t \in \mathcal{M}} \subset L^2(\hat{\G};l^2(I))$ is jointly measurable. The theorem now follows immediately from Theorem \ref{thm:frame} and Remark \ref{rem:CharParDet}.
\end{proof}

\begin{theorem} \label{thm:BrackRiesz}
In addition to the standing assumptions, suppose that $\G$ is discrete. For a countable family $\A \subset \H$ and constants $0 < A \leq B < \infty$, the following are equivalent:
\begin{enumerate}[(i)]
\item $E(\A)$ forms a Riesz basis for $S(\A)$ with bounds $A,B$.
\item For a.e.\ $\alpha \in \hat{\G}$, $\{ T \varphi(\alpha) : \varphi \in \A\}$ forms a Riesz sequence in $l^2(I)$ with bounds $A,B$.
\end{enumerate}
\end{theorem}

\begin{proof}
It follows from Theorem \ref{thm:Riesz} in the same way that Theorem \ref{thm:BrackFrame}
 followed from Theorem \ref{thm:frame}.
\end{proof}

For completeness, we mention the following combination of Lemma (2.8) and Proposition (5.1) in \cite{HSWW}.

\begin{prop}
In addition to the standing assumptions, assume that $\G$ is discrete. For a family $\A = (\theta_i)_{i\in I}\subset \H$, $E(\A)$ is an orthonormal sequence if and only if  $[\theta_i, \theta_j]= \delta_{i,j}$ a.e.
\end{prop}

\begin{rem}
Given a single vector $\psi \in \H$, we can replace $\H$ with $\langle \psi \rangle$ and take $\{\psi\}$ for our family of orthogonal generators. Then $T$ becomes the function $T_\psi \colon \langle \psi \rangle \to L^2(\hat{\G})$ from \eqref{eq:BrackIsomDef}. The range function $J_0 \colon \hat{\G} \to \{ \text{closed subspaces of }\C\}$ assigns $\C$ to every element of the set $\Omega_\psi$ from \eqref{eq:OmegaPsi}, and $\{0\}$ to every element of its complement. Taking $\A = \{\psi\}$ in Theorem \ref{thm:BrackFrame}, we see that the following are equivalent for constants $0 < A \leq B < \infty$:
\begin{enumerate}[(i)]
\item The orbit $(\pi(x) \psi)_{x\in \G}$ is a continuous frame for $\langle \psi \rangle$ with bounds $A,B$.
\item For a.e.\ $\alpha \in \Omega_\psi$, $A \leq [\psi,\psi](\alpha) \leq B$.
\end{enumerate}
This generalizes Theorem (5.7) of \cite{HSWW} for continuous frames. A similar analysis recovers \cite[Proposition (5.3)]{HSWW} from Theorem \ref{thm:BrackRiesz}.
\end{rem}

\begin{example}\label{examp:BrackReps}
Below are three prominent examples of dual integrable representations.

\noindent (i) If $\H_0$ is any separable Hilbert space, $\G$ acts on $L^2(\hat{\G}; \H_0)$ via the \emph{modulation representation} $\hat{\lambda}$ given by
\[ \hat{\lambda}(x) \varphi(\alpha) = \alpha(x)\cdot \varphi(\alpha). \]
This representation is dual integrable, and its bracket is given by the formula
\[ [\varphi, \psi](\alpha) = \langle \varphi(\alpha), \psi(\alpha) \rangle. \]

\medskip

\noindent (ii) Let $G$ be a second countable locally compact group. Any closed abelian subgroup $H \subset G$ acts on $L^2(G)$ by left translation. This representation is dual integrable, and the Zak transform gives a formula for the bracket. Indeed, \eqref{eq:ZakBrack} says that for $f,g \in L^2(G)$ and $\alpha \in \hat{H}$,
\[ [f,g](\alpha) = \langle (Zf)(\alpha^{-1}), (Zg)(\alpha^{-1}) \rangle_{L^2(H\backslash G)}. \]
When $G$ is abelian, the bracket can also be expressed in terms of the fiberization map. For $f,g \in L^2(G)$ and $\omega \in \hat{G}$, \eqref{eq:FibBrack} says that
\[ [f,g](\left. \omega \right|_H) = \langle ( \mathcal{T} f)(\omega^{-1}H^*), ( \mathcal{T} g)(\omega^{-1}H^*) \rangle_{L^2(H^*)}. \]
Theorems \ref{thm:HTIClass}, \ref{thm:HTIFrame}, and \ref{thm:HTIRiesz} can be recovered from Theorems \ref{thm:DualClas}, \ref{thm:BrackFrame}, and \ref{thm:BrackRiesz}, respectively.

\medskip

\noindent (iii) Let $G$ be a second countable LCA group with a closed subgroup $H$. Then $H \times H^*$ acts on $L^2(G)$ by translation and modulation. This representation is dual integrable, and the Zak transform gives a formula for the bracket, as follows. For any $f,g \in L^2(G)$, $\xi \in H$, and $\kappa \in H^*$, \eqref{eq:ZakTransMod} produces
\[ \langle f, L_\xi M_\kappa g \rangle_{L^2(G)} = \langle \tilde{Z} f, \tilde{Z} L_\xi M_\kappa g \rangle_{L^2(H\times H\backslash G)} = \int_{\hat{H}} \int_{H\backslash G} (\tilde{Z} f)(\alpha, Hx) \overline{ (\tilde{Z} g)(\alpha, Hx)}\cdot \alpha(\xi) \overline{\kappa(x)}\, d\mu_{H\backslash G} d\mu_{\hat{H}}(\alpha) \]
\[ = \int_{\hat{H}} \int_{H\backslash G} (\tilde{Z} f)(\alpha^{-1}, Hx) \overline{ (\tilde{Z} g)(\alpha^{-1}, Hx)}\cdot \overline{\alpha(\xi) \kappa(x)}\, d\mu_{H\backslash G} d\mu_{\hat{H}}(\alpha). \]
Since $H^* \cong (G/H)^\wedge$, Pontryagin Duality identifies $\widehat{H^*}$ with $H\backslash G$. For $Hx \in H\backslash G$, the corresponding character $X_{Hx} \in \widehat{H^*}$ is given by $X_{Hx}(\kappa) = \kappa(x)$. Thus,
\[ [f,g](\alpha,X_{Hx}) = (\tilde{Z}f)(\alpha^{-1}, Hx) \overline{ (\tilde{Z} g)(\alpha^{-1},Hx) }. \]
Theorems \ref{thm:TMIClass} and \ref{thm:TMIFrame} can be deduced from Theorems \ref{thm:DualClas} and \ref{thm:BrackFrame}.

\end{example}

We end with several equivalent conditions for dual integrability, continuing the list begun in \cite[Corollary 3.4]{HSWW}. The equivalence of (i) and (ii) below was essentially given there. From a philosophical perspective, the theorem below is the basis for our work on dual integrable representations, and the thread that connects Sections \ref{sec:frames}, \ref{sec:HTIStr}, and \ref{sec:DualInt}. We remind the reader that representations $\sigma$ and $\sigma'$ of $\G$ acting on Hilbert spaces $\H_\sigma$ and $\H_\sigma'$, respectively, are called \emph{unitarily equivalent} if there is a unitary $U\colon \H_\sigma \to \H_\sigma'$ such that $U \sigma(x) = \sigma'(x) U$ for all $x\in \G$.

\begin{theorem} \label{thm:DualEquiv}
For a representation $\sigma$ of $\G$, the following are equivalent:
\begin{enumerate}[(i)]
\item $\sigma$ is dual integrable, and the space on which it acts is separable.
\item There is a separable Hilbert space $\H_0$ for which $\sigma$ is unitarily equivalent to a subrepresentation of the modulation representation on $L^2(\G; \H_0)$.
\item There is a second countable locally compact group $G$ containing $\G$ as a closed subgroup, and $\sigma$ is unitarily equivalent to the left translation action of $\G$ on a $\G$-TI subspace of $L^2(G)$.
\end{enumerate}
\end{theorem}

\begin{proof}
That (iii) implies (i) is the content of Example \ref{examp:BrackReps}(ii). Corollary \ref{cor:BrackIsom} says that (i) implies (ii). Suppose (ii) holds. Without loss of generality, we may assume that $\H_0 = l^2(K)$ for some countable set $K$. Give $K$ the structure of a cyclic group, and let $G=\G \times K$. Let $\gamma \colon \G\backslash G \to G$ be the Borel section with fundamental domain $\gamma(\G \backslash G) = K \subset G$. Then the Zak transform is a unitary map $Z \colon L^2(G) \to L^2(\hat{\G};l^2(K))$ intertwining the translation action of $\G$ on $L^2(G)$ with modulation on $L^2(\hat{\G}; l^2(K))$. Following the unitary equivalence in (ii) with $Z^{-1}$ proves (iii).
\end{proof}

\medskip

%ACKNOWLEDGEMENTS====================================================

\section{Acknowledgements}
I am deeply indebted to Prof.\ Marcin Bownik, for all the reasons a Ph.D.\ student is normally indebted to a patient and insightful advisor, and in particular for reading the manuscript and giving helpful suggestions. I thank him especially for suggesting that Theorems \ref{thm:HTIFrame} and \ref{thm:HTIRiesz} should have measure theoretic generalizations. I also wish to thank Prof.\ Kenneth A.\ Ross, who put me on the spot in preliminary examinations with a question I could not answer. That question led to Theorem \ref{thm:RMeasIsom}, and whence to the paper before you. I thank him also for reading the manuscript. This research was funded in part by NSF grant DMS-1265711, and for that, I am grateful.

%BIBLIOGRAPHY===========================================================

\nocite{*}
\bibliographystyle{abbrv}
\bibliography{abelian}

\begin{thebibliography}{10}

\bibitem{ACHKM}
A.~Aldroubi, C.~Cabrelli, C.~Heil, K.~Kornelson, and U.~Molter.
\newblock Invariance of a shift-invariant space.
\newblock {\em J. Fourier Anal. Appl.}, 16(1):60--75, 2010.

\bibitem{AAG}
S.~T. Ali, J.-P. Antoine, and J.-P. Gazeau.
\newblock Continuous frames in {H}ilbert space.
\newblock {\em Ann. Physics}, 222(1):1--37, 1993.

\bibitem{ACP}
M.~Anastasio, C.~Cabrelli, and V.~Paternostro.
\newblock Extra invariance of shift-invariant spaces on {LCA} groups.
\newblock {\em J. Math. Anal. Appl.}, 370(2):530--537, 2010.

\bibitem{ACP2}
M.~Anastasio, C.~Cabrelli, and V.~Paternostro.
\newblock Invariance of a shift-invariant space in several variables.
\newblock {\em Complex Anal. Oper. Theory}, 5(4):1031--1050, 2011.

\bibitem{A}
A.~A. Arefijamaal.
\newblock The continuous {Z}ak transform and generalized {G}abor frames.
\newblock {\em Mediterr. J. Math.}, 10(1):353--365, 2013.

\bibitem{BHM}
D.~Barbieri, E.~Hern{\'a}ndez, and A.~Mayeli.
\newblock Bracket map for the {H}eisenberg group and the characterization of
  cyclic subspaces.
\newblock {\em Appl. Comput. Harmon. Anal.}, 37(2):218--234, 2014.

\bibitem{BHP}
D.~Barbieri, E.~Hern{\'a}ndez, and J.~Parcet.
\newblock {R}iesz and frame systems generated by unitary actions of discrete
  groups.
\newblock {\em Appl. Comput. Harmon. Anal.}
\newblock To appear.

\bibitem{BHP2}
D.~Barbieri, E.~Hern{\'a}ndez, and V.~Paternostro.
\newblock The {Z}ak transform and the structure of spaces invariant by the
  action of an {LCA} group.
\newblock Preprint.

\bibitem{B}
M.~Bownik.
\newblock The structure of shift-invariant subspaces of {$L^2({\bf R}^n)$}.
\newblock {\em J. Funct. Anal.}, 177(2):282--309, 2000.

\bibitem{B2}
M.~Bownik.
\newblock The structure of shift-modulation invariant spaces: the rational
  case.
\newblock {\em J. Funct. Anal.}, 244(1):172--219, 2007.

\bibitem{BR}
M.~Bownik and K.~A. Ross.
\newblock The structure of translation-invariant spaces on {LCA} groups.
\newblock {\em J. {F}ourier Anal. Appl.}
\newblock To appear.

\bibitem{CP}
C.~Cabrelli and V.~Paternostro.
\newblock Shift-invariant spaces on {LCA} groups.
\newblock {\em J. Funct. Anal.}, 258(6):2034--2059, 2010.

\bibitem{CP2}
C.~Cabrelli and V.~Paternostro.
\newblock Shift-modulation invariant spaces on {LCA} groups.
\newblock {\em Studia Math.}, 211(1):1--19, 2012.

\bibitem{C}
O.~Christensen.
\newblock {\em An introduction to frames and {R}iesz bases}.
\newblock Applied and Numerical Harmonic Analysis. Birkh\"auser Boston, Inc.,
  Boston, MA, 2003.

\bibitem{CMO}
B.~Currey, A.~Mayeli, and V.~Oussa.
\newblock Characterization of shift-invariant spaces on a class of nilpotent
  {L}ie groups with applications.
\newblock {\em J. Fourier Anal. Appl.}, 20(2):384--400, 2014.

\bibitem{BDR2}
C.~de~Boor, R.~A. DeVore, and A.~Ron.
\newblock Approximation from shift-invariant subspaces of {$L_2(\bold R^d)$}.
\newblock {\em Trans. Amer. Math. Soc.}, 341(2):787--806, 1994.

\bibitem{BDR1}
C.~de~Boor, R.~A. DeVore, and A.~Ron.
\newblock The structure of finitely generated shift-invariant spaces in
  {$L_2({\bf R}^d)$}.
\newblock {\em J. Funct. Anal.}, 119(1):37--78, 1994.

\bibitem{FG}
J.~Feldman and F.~P. Greenleaf.
\newblock Existence of {B}orel transversals in groups.
\newblock {\em Pacific J. Math.}, 25:455--461, 1968.

\bibitem{F}
G.~B. Folland.
\newblock {\em A course in abstract harmonic analysis}.
\newblock Studies in Advanced Mathematics. CRC Press, Boca Raton, FL, 1995.

\bibitem{Gr}
K.~Gr{\"o}chenig.
\newblock Aspects of {G}abor analysis on locally compact abelian groups.
\newblock In {\em Gabor analysis and algorithms}, Appl. Numer. Harmon. Anal.,
  pages 211--231. Birkh\"auser Boston, Boston, MA, 1998.

\bibitem{H}
H.~Helson.
\newblock {\em Lectures on invariant subspaces}.
\newblock Academic Press, New York-London, 1964.

\bibitem{H2}
H.~Helson.
\newblock {\em The spectral theorem}, volume 1227 of {\em Lecture Notes in
  Mathematics}.
\newblock Springer-Verlag, Berlin, 1986.

\bibitem{HSWW}
E.~Hern{\'a}ndez, H.~{\v{S}}iki{\'c}, G.~Weiss, and E.~N. Wilson.
\newblock Cyclic subspaces for unitary representations of {LCA} groups;
  generalized {Z}ak transform.
\newblock {\em Colloq. Math.}, 118(1):313--332, 2010.

\bibitem{HSWW2}
E.~Hern{\'a}ndez, H.~{\v{S}}iki{\'c}, G.~Weiss, and E.~N. Wilson.
\newblock The {Z}ak transform(s).
\newblock In {\em Wavelets and multiscale analysis}, Appl. Numer. Harmon.
  Anal., pages 151--157. Birkh\"auser/Springer, New York, 2011.

\bibitem{HR1}
E.~Hewitt and K.~A. Ross.
\newblock {\em Abstract harmonic analysis. {V}ol. {I}: {S}tructure of
  topological groups. {I}ntegration theory, group representations}.
\newblock Springer-Verlag, Berlin-G\"ottingen-Heidelberg, 1963.

\bibitem{HR2}
E.~Hewitt and K.~A. Ross.
\newblock {\em Abstract harmonic analysis. {V}ol. {II}: {S}tructure and
  analysis for compact groups. {A}nalysis on locally compact {A}belian groups}.
\newblock Springer-Verlag, New York-Berlin, 1970.

\bibitem{JL}
M.~S. Jakobsen and J.~Lemvig.
\newblock Reproducing formulas for generalized translation invariant systems on
  locally compact abelian groups.
\newblock {\em Trans. Amer. Math. Soc.}
\newblock To appear.

\bibitem{JM}
R.~Q. Jia and C.~A. Micchelli.
\newblock Using the refinement equations for the construction of pre-wavelets.
  {II}. {P}owers of two.
\newblock In {\em Curves and surfaces ({C}hamonix-{M}ont-{B}lanc, 1990)}, pages
  209--246. Academic Press, Boston, MA, 1991.

\bibitem{KS}
S.~Kaczmarz and H.~Steinhaus.
\newblock {\em Theorie Der Orthogonalreihen}.
\newblock Monografje matematyczne, Vol. 6. Chelsea, 1935.

\bibitem{Ka}
G.~Kaiser.
\newblock {\em A friendly guide to wavelets}.
\newblock Birkh\"auser Boston, Inc., Boston, MA, 1994.

\bibitem{KR}
R.~A. Kamyabi~Gol and R.~Raisi~Tousi.
\newblock A range function approach to shift-invariant spaces on locally
  compact abelian groups.
\newblock {\em Int. J. Wavelets Multiresolut. Inf. Process.}, 8(1):49--59,
  2010.

\bibitem{K}
A.~S. Kechris.
\newblock {\em Classical descriptive set theory}, volume 156 of {\em Graduate
  Texts in Mathematics}.
\newblock Springer-Verlag, New York, 1995.

\bibitem{N}
L.~Nachbin.
\newblock {\em The {H}aar integral}.
\newblock D. Van Nostrand Co., Inc., Princeton, N.J.-Toronto-NewYork-London,
  1965.

\bibitem{P}
B.~J. Pettis.
\newblock On integration in vector spaces.
\newblock {\em Trans. Amer. Math. Soc.}, 44(2):277--304, 1938.

\bibitem{RND}
A.~Rahimi, A.~Najati, and Y.~N. Dehghan.
\newblock Continuous frames in {H}ilbert spaces.
\newblock {\em Methods Funct. Anal. Topology}, 12(2):170--182, 2006.

\bibitem{RS2}
M.~Reed and B.~Simon.
\newblock {\em Methods of modern mathematical physics. {I}. {F}unctional
  analysis}.
\newblock Academic Press, New York-London, 1972.

\bibitem{RS}
H.~Reiter and J.~D. Stegeman.
\newblock {\em Classical harmonic analysis and locally compact groups},
  volume~22 of {\em London Mathematical Society Monographs. New Series}.
\newblock The Clarendon Press, Oxford University Press, New York, second
  edition, 2000.

\bibitem{R}
W.~Rudin.
\newblock {\em Real and complex analysis}.
\newblock McGraw-Hill Book Co., New York, third edition, 1987.

\bibitem{SW}
H.~{\v{S}}iki{\'c} and E.~N. Wilson.
\newblock Lattice invariant subspaces and sampling.
\newblock {\em Appl. Comput. Harmon. Anal.}, 31(1):26--43, 2011.

\bibitem{S}
T.~P. Srinivasan.
\newblock Doubly invariant subspaces.
\newblock {\em Pacific J. Math.}, 14:701--707, 1964.

\bibitem{W}
A.~Weil.
\newblock {\em L'int\'egration dans les groupes topologiques et ses
  applications}.
\newblock Actual. Sci. Ind., no. 869. Hermann et Cie., Paris, 1940.

\bibitem{W2}
A.~Weil.
\newblock Sur certains groupes d'op\'erateurs unitaires.
\newblock {\em Acta Math.}, 111:143--211, 1964.

\end{thebibliography}

\end{document}